\documentclass[a4paper,11pt]{amsart}
\usepackage[utf8]{inputenc}
\usepackage{amsmath, amssymb}
\usepackage{setspace}
\usepackage[left=2.4cm, right=2.4cm, top=2.2cm, bottom=2.2cm]{geometry}                 

\usepackage{graphicx} 
\usepackage{pgf,tikz}
\usetikzlibrary{arrows}
\usepackage{amssymb}
\usepackage{bbm}
\usepackage{pdfsync}
\usepackage{mathrsfs}
\usepackage{hyperref} 
\usepackage{verbatim} 
\usepackage{epstopdf}
\usepackage{xargs}
\usepackage{todonotes}
\usepackage[normalem]{ulem}

\newcommandx{\jow}[2][1=]{\todo[linecolor=orange,backgroundcolor=orange!25,bordercolor=orange,#1]{#2}}

\newcommandx{\mateus}[2][1=]{\todo[linecolor=blue,backgroundcolor=blue!25,bordercolor=blue,#1]{#2}}

\def\R{\mathbb{R}}
\def\P{\mathbb{P}}

\def\N{\mathbb{N}}
\def\Z{\mathbb{Z}}

\newcommandx{\eq}{\approxeq}

\newtheorem{theorem}{Theorem}

\newtheorem{remark}[theorem]{Remark}
\newtheorem{prop}[theorem]{Proposition}
\newtheorem{lemma}[theorem]{Lemma}

\newtheorem{example}[theorem]{Example}

\numberwithin{equation}{section}
\numberwithin{theorem}{section}

\begin{document}
\title[Stability of sphere packings]{A quantitative stability result for the sphere packing problem in dimensions 8 and 24}

\author{K\'aroly J. B\"or\"oczky, Danylo Radchenko, and Jo\~ao P. G. Ramos}
\begin{abstract}
We prove explicit stability estimates for the sphere packing problem in dimensions 8 and 24, showing that, in the lattice case, if a lattice is $\sim \varepsilon$ close to satisfying the optimal density, then it is, in a suitable sense,  close to the $E_8$ and Leech lattices, respectively. In the periodic setting, we prove that, under the same assumptions, we may take a large `frame' through which our packing locally looks like $E_8$ or $\Lambda_{24}.$  

Our methods make explicit use of the magic functions constructed in \cite{Via17} in dimension 8 and in \cite{CKMRV17} in dimension 24, together with results of independent interest on the abstract stability of the lattices $E_8$ and $\Lambda_{24}.$ 
\end{abstract}

\maketitle

\tableofcontents

\section{Introduction}

Consider a disjoint collection of unit balls in the Euclidean space $\R^n.$ A natural question regarding these arrangements of points is the following: for a fixed dimension $n \in \N$, what is the \emph{maximal} density among all such configurations? Moreover, what are the configurations attaining the maximum? Such a configuration is generally called a \emph{sphere packing}, and the aforementioned problems are known as versions of the (Euclidean) \emph{sphere packing problem}, a pillar of modern metric geometry.

This problem has eluded mathematicians for centuries. Indeed, even if we restrict our attention to lattice packings, the largest density is only known in a few dimensions: for $n=2$ by the work of Lagrange \cite{Lag73} from 1773, for $n=3$ by Gauss \cite{Gau31} from 1831, for $n=4,5$ by Korkin and Zolotarev \cite{KoZ72,KoZ77} from around 1875, for $n=6,7,8$ by Blichfeldt \cite{Bli34} from 1934,
and for $n=24$ by Cohn, Kumar \cite{CoK09} from 2009.

Concerning the densest \emph{general} packings of congruent balls in $\R^n$, even less is known. Around 1900, already Minkowski observed that saturated packings of congruent balls in $\R^n$ --- that is, when no extra ball can be inserted without violating the packing property --- have density at least $2^{-n}$. This trivial lower bound has been only slighly improved asymptotically even after significant efforts since Minkowski's times. The current record is due to Venkatesh \cite{Ven13}, who proved that there exist lattice packings of spheres for all sufficiently large $n$ whose density is at least $65{,}963n\,2^{-n}$, and that for infinitely many values of~$n$ there exist lattice packings of spheres whose density is at least $\frac{n\log\log n}{2^{n+1}}$. 
On the other hand, Kabatyanskii and Levenshtein \cite{KaL78} still hold the best asymptotic upper bound $2^{-0.599\ldots n+o(n)}$ to date for the density of packings of congruent spheres in $\R^n$ using the linear programming bound for sphere packings on $S^{n-1}$. This bound was slighly improved by the pioneering work of Cohn and Elkies \cite{CoE03}, which applied the linear programming method directly in the Euclidean setting. In particular, these bounds, together with numerical computations, yielded almost exact upper bounds in dimensions $n=8$ and $n=24$. As the following list shows, the maximal density of sphere packings is only known in dimensions when a lattice with very special metric properties is available $(n= 8, 24)$, or where low dimensionality allows for more direct combinatorial arguments, possibly combined with computational tools $(n=2, 3)$.

\begin{enumerate}
	\item[$n=2:$] The so-called hexagonal lattice (generated by two side vectors of an equilateral triangle) gives the optimal packing of circles in two dimensions according to Thue \cite{Thu92,Thu10}. However, Thue's proof lacked a compactness argument which was later provided by L. Fejes Toth \cite{FTL40} and Segre and Mahler \cite{SeM44}.
	\item[$n=3:$] The famous \emph{Kepler's conjecture} ({\it cf.}~\cite{Kep11}) stated that the best density in three dimensions is achieved by the face centered cubic (fcc) lattice; namely, by piling balls as a pile of oranges in the supermarket shelf (or alternatively, as cannonballs in a pirate ship). 
	This was proven by Hales~\cite{Hal05} with parts of the argument relying on computer calculations, and a formally verified proof was given by Hales {\it et al} \cite{Hal17} in 2017. 
	\item[$n=8:$] In 2016, Maryna Viazovska \cite{Via17} made a major breakthrough 
	by solving the sphere packing problem in dimension $8$, confirming the optimality of the $E_8$ lattice (see Theorem~\ref{E8} below). The paper \cite{Via17} used the theory of modular forms in an ingenious combination with the linear programming bound of Cohn and Elkies~\cite{CoE03}. We note that $E_8$ is the unique even unimodular lattice in 8 dimensions according to Mordell \cite{Mor38} (see for example Griess \cite{Gri03} and Elkies \cite{Elkies} for simpler proofs).
	\item[$n=24:$] Cohn, Kumar, Miller, Radchenko, Viazovska \cite{CKMRV} established the optimality of the Leech lattice $\Lambda_{24}$ in $24$ dimensions (see Theorem~\ref{Leech} below). The Leech lattice is the unique even unimodular lattice in $\R^{24}$ whose minimal length is $2$.
\end{enumerate}


Our focus in this manuscript will be on the \emph{stability question} for sphere packings in dimensions 8 and 24. That is, if a packing $\Xi + B^n$, $n \in\{8,24\},$ is \emph{almost optimal} in the aforementioned results for dimensions 8 and 24, must it be in some sense close to the optimal packing --- $E_8$ for $n=8$ and $\Lambda_{24}$ for $n=24$? 

In order to start our discussion we recall the exact statement of the main result in \cite{Via17}.

\begin{theorem}[Viazovska]
	\label{E8}
	If $\Xi+\frac{\sqrt{2}}2\,B^8$ is a periodic packing 
	in $\R^8$, then its center density is at most $1$, with equality  if and only if $\Xi$ is congruent to $E_8$.
\end{theorem}

Following the 8-dimensional contribution, the problem in dimension 24 was settled days later with similar techniques
by Cohn, Kumar, Miller, Radchenko, Viazovska \cite{CKMRV}. 

\begin{theorem}[Cohn, Kumar, Miller, Radchenko, Viazovska]
	\label{Leech}
	If $\Xi+B^{24}$ is a periodic packing in~$\R^{24}$, then its center density is at most $1$, with equality  if and only if $\Xi$ is congruent to the Leech lattice $\Lambda_{24}$.
\end{theorem}


Regarding the stability question, we note that G. Fejes T\'oth \cite{FTG01} gave a stability version of the optimality of the hexagonal lattice packing of congruent circular disks in $\R^2$, see also Caglioti, Golse, Iacobelli \cite{Cag18}. The goal of this note is to prove explicit quantitative stability versions of Theorems~\ref{E8} and \ref{Leech}. We start by stating  a couple of results in the realm of lattice packings, in which case we have rather precise statements. Recall that for lattices $L \subset \R^n,$ the center density is $1/\det L$. We also denote by $\lambda(L)$ the length of the shortest non-zero vector in $L$. 

\begin{theorem}
\label{E8stab}
If $L$ is a lattice in $\R^8$ such that $\lambda(L)\geq \sqrt{2}$ and $\det L \le 1 + \varepsilon$  for $0<\varepsilon<\varepsilon_8$, then there exist $\Phi\in O(8)$,
a basis $w_1,\ldots,w_8$  of $E_8$ of length at most $2$ and a basis $u_1,\ldots,u_8$ of~$L$ 
such that
$$
\|u_i-\Phi w_i\|\leq c_8\varepsilon \mbox{ \ for $i=1,\ldots,8$}\,,
$$
where $\varepsilon_8, c_8 > 0$ are computable parameters. 
\end{theorem}

\begin{theorem}
\label{Leechstab}
If $L$ is a lattice 
in $\R^{24}$ such that $\lambda(L)\geq 2$ and $\det L \le 1+\varepsilon$ for
$0<\varepsilon<\varepsilon_{24}$, then there exist $\Phi\in O(24)$,
basis $w_1,\ldots,w_{24}$ of $\Lambda_{24}$ of length at most $2^{140}$ and a basis 
$u_1,\ldots,u_{24}$ of $L$ such that
$$
\|u_i-\Phi w_i\|\leq c_{24}\varepsilon \mbox{ \ for $i=1,\ldots,24$}
$$
where $\varepsilon_{24}, c_{24} > 0$ are computable parameters. 
\end{theorem}

Theorems \ref{E8stab} and \ref{Leechstab} follow by adapting the arguments from \cite{CoK09}, with the aid of some techniques from our present manuscript. We outline a proof of these results at the end of Section \ref{sec:ProofLattice}. 

Prior to that, we provide a proof of a weaker version of Theorems \ref{E8stab} and \ref{Leechstab}, with $\varepsilon^{1/2}$ in place of $\varepsilon$ on the right-hand side of the conclusion (see Proposition \ref{WeakerStab}). The reason for that is twofold: the first of them is that this argument, in spite of not obtaining the linear bound in terms of $\varepsilon,$ is significantly shorter than repeating the full argument required for the proof of Theorems \ref{E8stab} and \ref{Leechstab}. The other reason is that the techniques used for the proof of the weaker version are directly related to the proofs of the next results, devoted to proving stability in the setting of periodic packings. To handle that case, we need some more notions.

For compact sets $\Xi_1,\Xi_2\subset \R^n$, their \emph{Hausdorff distance} is
$$
d_H(\Xi_1,\Xi_2)=\min_{r\geq 0}\{\Xi_1\subset \Xi_2+r B^n\mbox{ and }\Xi_2\subset \Xi_1+r B^n\},
$$
which is a metric. We observe that if $\Xi_1,\Xi_2$ are finite such that
$\Xi_1+\varrho\,B^8$ and $\Xi_2+\varrho\,B^8$ are packings for $\varrho>0$ and 
$d_H(\Xi_1,\Xi_2)<\varrho$, then $\Xi_1$ and $\Xi_2$ have the same cardinality.

Let $\Pi$ be a measurable set of $ \R^n$ that is periodic with respect to a lattice $L$. Denote by $\tilde{V}$ the (induced) Lebesgue measure on $\R^n/L.$ Then we say that $\Pi$ holds for $x\in \R^n$ with probability  $p\in[0,1]$ if we have
$$
p=\frac{\tilde{V}(\Pi/L)}{\det L}=
\lim_{r\to \infty}\frac{V(\{x\in rK:\,\mbox{$x \in$ $\Pi$}\})}{V(rK)}
$$
for any convex body $K$. We observe that, since $\Pi$ is $L-$periodic, we obtain the same probability starting from any other lattice $L'\subset L$. 

\begin{theorem} 
\label{PeriodicPacking8stab0}
There exist explicitly computable values $\tilde{\varepsilon}_8, \tilde{c}_8 >0$ such that for 
$\varepsilon\in(0,\tilde{\varepsilon}_8)$ and
$R_\varepsilon = \frac{|\log \varepsilon|}{\log|\log \varepsilon|}$, the following properties hold:

If a periodic packing $\Xi+\frac{\sqrt{2}}2\,B^8$ of balls  has center density at least $1-\varepsilon$ for 
 $\varepsilon\in(0,\tilde{\varepsilon}_8)$, and $K$ is a centered convex body containing a ball of radius $R_\varepsilon$ and
having diameter ${\rm diam}\,K\leq 2^{20}R_\varepsilon$, then 
with probability at least $1-\tilde{c}_8R_\varepsilon^{-\frac{1}2}$, $x\in\R^8$ satisfies that 
$\# (\Xi\cap (x+K))\geq (1-\tilde{c}_8R_\varepsilon^{-\frac{1}2})V(K)$ and
$$
d_H\Big(\Phi Z,\Xi\cap (x+K)\Big)\leq \tilde{c}_8 e^{\frac{1}{\tilde{c}_{8}} R_{\varepsilon}} \varepsilon^{\frac{1}{4}}
\mbox{ \ \ where $Z\subset E_8$ and $\Phi$ is an isometry of $\R^8$.}
$$ 
Moreover, the set $Z$ has small gaps when compared to a localized version of $E_8,$ in the sense that 
\[
\left|\frac{\#(\Phi Z)}{\#((\Phi E_8) \cap (x+K))} - 1 \right| \le \frac{C_8}{\sqrt{R_{\varepsilon}}},
\]
where $C_8 > 0$ is an absolute computable constant. 
\end{theorem}

\begin{theorem} 
	\label{PeriodicPacking24stab0}
There exist explicitly computable values $\tilde{\varepsilon}_{24}, \tilde{c}_{24} >0$ such that for 
	$\varepsilon\in(0,\tilde{\varepsilon}_{24})$ and $R_{\varepsilon} = \frac{|\log \varepsilon|}{\log |\log \varepsilon|}$, the following properties hold.
	
	If a periodic packing $\Xi+B^{24}$ of balls  has center density at least $1-\varepsilon$ for 
	$\varepsilon\in(0,\tilde{\varepsilon}_{24})$, and $K$ is a centered convex body containing a ball of radius $R_\varepsilon$ and
	having diameter ${\rm diam}\,K\leq 2^{140}R_\varepsilon$, then 
	with probability at least $1-\tilde{c}_{24}R_\varepsilon^{-\frac{1}2}$, an $x\in\R^{24}$ satisfies that 
	$\# (\Xi\cap (x+K))\geq (1-\tilde{c}_{24}R_\varepsilon^{-\frac{1}2})V(K)$ and
	$$
	d_H\Big(\Phi Z,\Xi\cap (x+K)\Big)\leq 
\tilde{c}_{24}  e^{\frac{1}{\tilde{c}_{24}} R_{\varepsilon}} \varepsilon^{\frac{1}{4}}
	\mbox{ where $Z\subset \Lambda_{24}$ and $\Phi$ is an isometry of $\R^{24}$. }
	$$ 
 Moreover, the set $Z$ has small gaps when compared to a localized version of $\Lambda_{24},$ in the sense that 
	\[
	\left|\frac{\#(\Phi Z)}{\#((\Phi \Lambda_{24}) \cap (x+K))} - 1 \right| \le \frac{C_{24}}{\sqrt{R_{\varepsilon}}},
	\]
	where $C_{24} > 0$ is an absolute computable constant. 
\end{theorem}

\begin{remark} Observe that, in spite of the perhaps cumbersome notation, in both Theorems \ref{PeriodicPacking8stab0} and \ref{PeriodicPacking24stab0} we have that $d_H\Big(\Phi Z,\Xi\cap (x+K)\Big)\lesssim \varepsilon^{\frac{1}{4}-o(1)}$ as $\varepsilon$ goes to zero. It is an interesting question whether this may be improved to $d_H\Big(\Phi Z,\Xi\cap (x+K)\Big)\lesssim \varepsilon^{\frac{1}{4}}$ for all $\varepsilon > 0.$ 
\end{remark}

Paraphrasing, these results tell us that up to adjusting the frame to convex sets of diameter $\sim R_{\varepsilon},$ an almost optimal packing `almost always' looks like a lattice packing, barring removing a small set of centers. The choice of $R_{\varepsilon}$ here was taken so that the bound on the Hausdorff distance is smaller than any power $\varepsilon^{\alpha}, \, \alpha < \frac{1}{4}.$ Such a choice, as one will see from the proof, is not rigid: one could either further adjust the frame by `zooming in' in order to obtain a slightly better estimate on such a Hausdorff dimension, or `zoom out' in order to have a stability result at a higher scale, albeit with less precision. 

In order to prove these results, we employ several mechanisms: the first is, as in the lattice packing case, the use of `almost vanishing' of the magic functions $g_8$ and $g_{24}.$ This gives us, implicitly, that most differences of points in the periodic configuration are close to zeros of the magic function. The differences which are not are relatively few, and these give rise to a small set of `bad' points, which is the one which we discard, as stated in the proofs of Theorems \ref{PeriodicPacking8stab0} and \ref{PeriodicPacking24stab0}. For the remaining `good' points, the proximity to zeros of the magic function imposes rigidity conditions, just as in the lattice case, which allow us to explicitly construct a (rotated) basis of $E_8$ and $\Lambda_{24}$ close to a basis of the original lattice, after a suitable translation.

One readily notices the differences between Theorems \ref{E8stab}, \ref{Leechstab} and Theorems \ref{PeriodicPacking8stab0}, \ref{PeriodicPacking24stab0}. In particular, one may wonder whether the weaker nature in the latter results cannot be dropped. In the next example we show that one can't expect a more precise statement about almost optimal periodic packings in $\R^8$ than Theorem~\ref{PeriodicPacking8stab0}, and $R_\varepsilon$ is at most of order $1/\varepsilon$. An entirely analogous construction is possible in dimension $24.$ 

\begin{example}\label{example optimal} 
Let $\Psi\in O(8)$ be such that $E_8\cap \Psi E_8=\{0\}$.
For any large positive integer $R$, we consider $\varepsilon=\frac1{R}$, $\Lambda=R^2\Z^8$, and let
$$
\{v_1,\ldots,v_{R^8}\}=[0,R-1]^8\cap \Z^8.
$$
For $i=1,\ldots,R^8$, let
$$
W_i=Rv_i+[0,R]^8.
$$
We construct the packing $\Xi_i+\frac{\sqrt{2}}2\,B^8\subset W_i$ if $R^7<i\leq R^8$. First,
if $R^7<i\leq R^8$, then let
$$
\Xi'_i=\left\{\begin{array}{l}
	\left\{x:x+\frac{\sqrt{2}}2\,B^8\subset W_i\mbox{ and }x-Rv_i\in E_8\right\}
	\mbox{ if the sum of coordinates of $v_i$ is even}\,,\\[1ex]
	\left\{x:x+\frac{\sqrt{2}}2\,B^8\subset W_i\mbox{ and }x-Rv_i\in \Psi E_8\right\}
	\mbox{ if the sum of coordinates of $v_i$ is odd}\,,
\end{array}\right.
$$
and let $\Xi_i$ be obtained from $\Xi'_i$ by deleting arbitrary $R^7$ points from $\Xi'_i$.
Our packing is $\Xi+\frac{\sqrt{2}}2\,B^8$, where
$$
\Xi=\Lambda+\bigcup_{R^7<i\leq R^8}\Xi_i.
$$
Then the center density of the periodic packing is at least $1-O(\varepsilon)$, the packing contains no patch of size larger then $R=1/\varepsilon$ that can be reasonably well approximated by a part of the sphere packing by $E_8$, and the packing may contain holes of size $R^{7}$.
\end{example}

The rest of the manuscript is organized as follows. In Section~\ref{sec:prelim}, we collect some preliminary results, ranging from properties of lattices to claims on the functions $g_8$ and $g_{24}.$ In Section \ref{sec:ProofLattice}, we prove Theorems~\ref{E8stab} and~\ref{Leechstab} on stability for lattice packings, and in Section~\ref{sec:Periodic} we prove Theorems~\ref{PeriodicPacking8stab0} and~\ref{PeriodicPacking24stab0}. Finally, Section~\ref{secbin} deals with natural generalizations of the main results to the context of bin packings and general packings. 

\section{Preliminaries}\label{sec:prelim} 

\subsection{Definitions and properties of packings} 
For $r>0$, $\Xi\subset \R^n$ and $D=rB^n$,
we recall that $\Xi+D$ is a \emph{sphere packing} if ${\rm int}\,(x+D)\cap {\rm int}\,(y+D)=\emptyset$ for different $x,y\in \Xi$,
and the upper density of the packing is defined as
$$
\delta_{\rm upp}(\Xi,rB^n) = \limsup_{R \to \infty} \frac{V((z+R\,B^n) \cap (\Xi +D))}{V(R\,B^n)}
$$
for any fixed $z\in\R^n$. A packing is called \emph{periodic} if $\Xi$ is invariant under a lattice $\Lambda$, and hence if~$\Xi$ represents $N$ cosets with respect to $\Lambda$, then one has
$$
\delta_{\rm upp}(\Xi,D) = \lim_{R \to \infty} \frac{V(R\,B^n \cap (\Xi +D))}{V(R\,B^n)}=\frac{N\cdot V(D)}{\det\Lambda}.
$$ 
In the case when $\Xi=\Lambda$, thus $N=1$, we call the packing a \emph{lattice packing}. The classical paking density $\delta(B^n)$ is thus defined as the supremum of $\delta_{\rm upp}(\Xi,B^n)$ over all packings 
$\Xi+B^n$, which equals the supremum of $\delta_{\rm upp}(\Xi,rB^n)$ over all packings 
$\Xi+rB^n$ for any fixed $r>0$. According to Groemer \cite{Gro63}, we may define $\delta(B^n)$ to be the supremum of the densities over all \emph{periodic} packings, and in fact, there exists a (possibly non-periodic) packing $\Xi_0+B^n$ such that for any convex body $K$ (compact convex set with non-empty interior), we have
$$
\delta(B^n)=\delta_{\rm upp}(\Xi_0,B^n) = \lim_{R \to \infty} \frac{V(R\,K \cap (\Xi_0 +B^n))}{V(R\,K)}.
$$ 

Let $D$ be a Euclidean ball in $\R^n$. For a periodic packing $\Xi+D$, it is sometimes more convenient to work with the center density
$$
\Delta(\Xi,D) =\frac{\delta(\Xi,D)}{V(D)} = \lim_{R \to \infty} \frac{\#(\Xi\cap RB^n)}{V(R\,B^n)}=\frac{N}{\det\Lambda},
$$
where $\Xi$ is periodic with respect to the lattice $\Lambda$ and $\Xi$ represents $N$ cosets with respect to $\Lambda$. Similarly, for a general packing $\Xi+D$, the upper center density is defined as
$$
\Delta_{\rm upp}(\Xi,D)=\frac{\delta_{\rm upp}(\Xi,D)}{V(D)}= 
\limsup_{R \to \infty} \frac{\#(\Xi\cap RB^n)}{V(R\,B^n)}.
$$

A cornerstone concept used to study sphere packings is that of a \emph{lattice}, which we define to be a discrete subgroup $\Lambda$ of $\R^n$ whose $\R$-linear span is $\R^n$. In this case, there exist $u_1,\ldots,u_n$ such that $\Lambda=\sum_{i=1}^n\Z u_i$, and any such $n$-tuple of vectors in $\Lambda$ is called a basis of~$\Lambda$. In addition, $\det \Lambda$ is the common absolute value of the determinant $\det[u_1,\ldots,u_n]$ for all bases of $\Lambda$.

For a lattice $\Lambda$ in $\R^n$, we use $\lambda(\Lambda)$ to denote the minimal length of non-zero vectors. That is, 
\[
\lambda(\Lambda) = \min\{ \|v\| \colon v \in \Lambda\backslash\{o\} \}.
\]
Let $\Lambda^*=\{z\in\R^n:\,\langle z,x\rangle\in\Z\mbox{ for }x\in \Lambda\}$ denote the dual lattice. For the dual lattice, we have $\det \Lambda^*=1/\det \Lambda$.

We note that if $\Lambda$ is a lattice in $\R^n$, then $\Lambda+\frac{\lambda(\Lambda)}2\,B^n$ is a packing where 
$B^n$ is the unit Euclidean ball centered at the origin.
We call a lattice $\Lambda$ \emph{unimodular} if $\det \Lambda=1$ and \emph{even}  if $\|x\|^2\in 2\Z$ for any $x\in \Lambda$. If $\Lambda$ is even and unimodular, then
\begin{equation}
	\label{xyxxyy}
	\langle x,y\rangle=\frac{\|x+y\|^2-\|x\|^2-\|y\|^2}2\in\Z
\end{equation}
for any $x,y\in \Lambda$; therefore, the dual lattice $\Lambda^*$ coincides with $\Lambda$. 

It is known that a lattice $\Lambda$ in $\R^n$ possesses a so-called Lenstra-Lenstra-Lov\'asz-reduced (or LLL-reduced) basis $u_1,\ldots,u_n$ such that
the product of lengths is at most a constant times $\det \Lambda$; in other words, there is a basis $u_1,\dots,u_n$ of $\Lambda$ such that 
\begin{equation}
	\label{LLL}
	\|u_1\|\cdot\ldots\cdot\|u_n\|\leq 2^{\frac{n(n-1)}4}\det \Lambda.
\end{equation}
We will simply call such bases \emph{LLL bases}. 

\subsection{Properties of the functions $g_8$ and $g_{24}$} We recall first the following core idea in the proof of Theorem~3.1 in Cohn, Elkies~\cite[p.~694]{CoE03}, which will be of utter importance in our proofs: 
\begin{prop}
If $\Lambda+S+B^n$ is a packing of balls, where $\Lambda$ is a lattice, the set $S$ has (finite) cardinality $\#S = N$ and $\Lambda\cap (S-S)=\{o\}$, and $g_n:\R^n\to \R$ is a Schwarz function satisfying
$g_n(o) = \widehat{g_n}(o) = 1, \, \widehat{g_n} \ge 0,$ then
\begin{equation}
	\label{CohnElkies}
	N+\sum_{z\in\Lambda,\;v,w\in S\atop
		z+v-w\neq o} g_n(z+v-w)\geq \frac{N^2 }{\det \Lambda}.
\end{equation}
\end{prop}

This proposition follows at once from the Poisson summation formula for (translated versions of) the lattice $\Lambda.$ Moreover, if one assumes that $g_n(x) \leq 0$ for $|x|\geq \varrho$, then
$\Delta(\Lambda + S) \le (\varrho/2)^n.$ 

We quickly recall the breakthrough from both works \cite{Via17,CKMRV17}: if $n=8$, then Viazovska constructed a radial Schwartz function $g_8:\R^8 \to \R$ which satisfies the hypotheses above with $\varrho={\sqrt{2}}.$ Moreover, $g_8$ satisfies that $g_8(\sqrt{2k})=0$ for integers $k\geq 1$, $g_8$ has a simple root at $\sqrt{2}$, and $g_8$ has a double root at $\sqrt{2k}$ when $k\geq 2$ (as $g_8$ is radial, we write $g_8(x)$ as  $g_8(|x|)$, and we will do similarly in the case of $g_{24}$).
For $n=24$, Cohn, Kumar, Miller, Radchenko and Viazovska subsequently constructed a radial function $g_{24}:\R^{24}\to \R$ for $\varrho=2$, which additionally satisfies $g_{24}(\sqrt{2k})=0$
for integers $k\geq 2$, has a simple root at $2$, and
$g_{24}$ has a double root at $\sqrt{2k}$
when $k\geq 3$. 

We are going to need the following properties regarding the functions $g_8$ and $g_{24}:$ 

\begin{lemma}
\label{tkg8}
For  $R>2$, there exists $\alpha_8(R),\alpha_{24}(R)>1$ such that the following hold:
\begin{enumerate} 
	\item if $\sqrt{2}\leq t\leq R$, then there exists an integer $k\geq 1$ such that 
$$
|t^2-2k|\leq \alpha_8(R)\sqrt{|g_8(t)|};
$$
   \item if $2 \le t \le R,$ then there exists an integer $k \ge 2$ such that 
$$
|t^2-2k|\leq \alpha_{24}(R)\sqrt{|g_{24}(t)|}.
$$
\end{enumerate}
Moreover, there exists an absolute constant $c > 0$ such that one may take $\alpha_8(R) \le \rho_0 R^{\frac{3}{2}} e^{\frac{5}{4} \pi R}$ and $\alpha_{24}(R) \le \rho_1 R^{\frac{11}{2}} e^{\frac{5\sqrt{2}}{2} \pi R}, \, \forall \, R > 1,$ where $\rho_0,\rho_1 > 0$ are two absolute computable constants.  
\end{lemma}

\begin{proof}  We provide a proof for $g_8,$ and later we indicate the places where slight modifications are needed in order to cover the $g_{24}$ case. 
	
We start by observing  that we have, for $r > \sqrt{2},$ 
\[
g_8(r) = \frac{\pi}{2160} \sin^2(\pi r^2/2) \int_0^{\infty} A_8(t) e^{- \pi r^2 t} \, dt,
\]
where $A_8(t) < 0$ for $t \in (0,+\infty).$ Notice, from the proof of Theorem~4 in Viazovska's paper, that, letting $a_0(t) = - \frac{368640}{\pi^2} t^2 e^{-\pi/t},$ then 
\[
|A_8(t) - a_0(t)| \le 2 \left(t^2 + \frac{36}{\pi^2}\right) \sum_{n=2}^{\infty} e^{2 \sqrt{2}\pi \sqrt{n}} e^{-\pi n/t} \le  2 \left(t^2 + \frac{36}{\pi^2}\right)  \frac{e^{16 \pi } e^{-2\pi/t}  }{1- e^{2 \sqrt{2} \pi} e^{-\pi/t}}. 
\]
For $t < \frac{1}{10^7},$ this implies the bound 
\[
|A_8(t) - a_0(t)| \le 4 e^{16 \pi}  \left(t^2 + \frac{36}{\pi^2}\right) e^{-2\pi/t} \le \frac{t^2 e^{-\pi/t}}{\pi^2},
\] 
which implies that $A_8(t) \le - \frac{a_0(t)}{2}$ on the interval $(0,10^{-7}).$ Thus: 
\begin{equation}\label{eq:bound-modular-A_8} 
\int_0^{\infty} A_8(t) e^{- \pi r^2 t} \, dt \le - \frac{368640}{2\pi^2} \int_0^{10^{-7}} t^2 e^{-\pi/t} e^{-\pi r^2 t} \, dt. 
\end{equation}
If $r^{-2} > \frac{10^{-14}}{4},$ we bound the right-hand side of \eqref{eq:bound-modular-A_8} by $ - \frac{368640}{2\pi^2}\cdot \int_0^{10^{-7}} e^{-\pi/t} e^{-\pi t (4/10^{14})}.$ If, on the other hand, $r^{-2} < \frac{10^{-14}}{4},$ we see that $e^{-\pi \left( \frac{1}{t} + r^2 t \right)}$  attains its maximum for $t = 1/r,$ and thus, by restricting $t$ to the interval between $\frac{1}{2r}$ and $\frac{2}{r},$  \eqref{eq:bound-modular-A_8} may be bounded by $-c_0 \frac{e^{-\frac{5}{2}\pi r}}{r^3}.$ Therefore, 
\[
c_1 R^3 e^{\frac{5}{2}\pi R} |g_8(r)| \ge \sin^2(\pi r^2/2) \ge \frac{1}{5} \min_{k \in \Z} |r^2 - 2k|^2, 
\]
where $c_1$ is an explicitly computable constant, and $r \in (\sqrt{2}, R).$  Thus, one may take $\alpha_8(R) = c_2 R^{\frac{3}{2}} e^{\frac{5}{4} \pi R}$ for some absolute (and computable) constant $c_2 > 0.$ 

For the case of $g_{24},$ such an explicit bound is not readily available. What holds, on the other hand, by Lemma~A.1 and Section~4 in \cite{CKMRV17}, is that $A_{24}(t) \le \frac{\pi}{28304640} t^{10} \varphi(i/t),$ where $\varphi$ is a certain quasimodular form, such that the $n$-th coefficient $c(n)$ of the  $q$-expansion of $\varphi \Delta^2$ satisfies
\[
|c(n)| \le 513200655360 (n+1)^{20}.
\] 
From that, we see that 
\begin{equation}\label{eq:bound-modular-A_24} 
|(\varphi \cdot \Delta^2)(i/t) + 3657830400 e^{-4 \pi/t}| \le \sum_{n=3}^{\infty}  513200655360 (n+1)^{20} e^{- 2 n \pi /t}.
\end{equation}
Bounding $(1+n) \le e^n$ and using similar estimates as in the eight-dimensional case, we obtain 
\[
|(\varphi \cdot \Delta^2)(i/t) + 3657830400 e^{-4 \pi/t}| \ \le 513200655360 e^{60} \frac{e^{-6 \pi /t}}{1- e^{20} e^{-2 \pi/t}}. 
\]
For $t_0>0$ small enough, the right-hand side above may be bounded by $\frac{1}{2}e^{-6\pi/t}$ whenever $t < t_0.$  Thus, $(\varphi \cdot \Delta^2) (i/t) \le - \frac{3657830400}{2} e^{-4 \pi/t}$ for $0<t < t_0.$ On the other hand, observe that 
\begin{align*}
\log \Delta(i/t) & \ge - \frac{2 \pi}{t} + 24 \sum_{n=1}^{\infty} \log(1- e^{-2\pi n/t}) \cr 
 					  & \ge - \frac{2 \pi}{t} - 36 \sum_{n=1}^{\infty} e^{-2\pi n/t}  \ge - \frac{2 \pi}{t} - 100.
\end{align*}
whenever $t \in (0,t_0),$ and thus $\Delta^2(i/t) \ge e^{-2\pi/t} e^{-100}.$ From that, we readily obtain that $ \varphi(i/t) \le - c_2 e^{-2\pi/t}$ for $t \in (0,t_0),$ where $c_2 = \frac{3657830400}{2} e^{100}.$ One concludes that for $c_3 = \frac{\pi}{28304640} c_2,$ one has $A_{24}(t) \le -c_3 t^{10} e^{-2\pi/t}$ if $t \in (0,t_0).$ The same analysis as in the eight-dimensional case plus the formula 
\[
g_{24}(r) = \sin^2(\pi r^2/2) \int_0^{\infty} A_{24}(t) e^{-\pi r^2 t} \, dt,
\]
show that we may take $\alpha_{24}(R) = c_4 R^{\frac{11}{2}} e^{\frac{5\sqrt{2}}{4} \pi R}$ for some absolute computable constant $c_4 > 0.$ 
\end{proof}

\subsection{Approximating a basis of $\R^n$} For linearly independent $x_1,\ldots,x_i\in \R^n$, 
let us write $\det_{i}(x_1,\ldots,x_{i})$ 
to denote the determinant of $x_1,\ldots,x_{i}$ in ${\rm lin}\{x_1,\ldots,x_{i}\}$; namely,
$$
{\rm det}_{i}(x_1,\ldots,x_{i})=\sqrt{{\rm det} [x_1,\ldots,x_{i}]^t[x_1,\ldots,x_{i}]}.
$$
In addition, if $x_1,\ldots,x_n\in \R^n$, then we write $x_1,\ldots,\hat{x}_i,\ldots,x_n$ to denote the list 
of $n-1$ elements where $x_i$ is excluded.

Our first statement of this subsection provides a condition for a basis $u_1,\ldots,u_n$ of $\R^n$ to be approximately orthonormal. 

\begin{lemma}
\label{dualbasis}
Let $u_1,\ldots,u_n\in\R^n$ be linearly independent, and let $d,D,\varrho>0$ such that
$d\leq \|u_i\|\leq D$ for $i=1,\ldots,n$, and
$\prod_{i=1}^n \|u_i\|\leq \varrho\Big|\det[u_1,\ldots,u_n]\Big|$.
If $u_1^*,\ldots,u_n^*$ is the dual basis, i.e.,
$\langle u_i,u_j^*\rangle=0$ if $i\neq j$ and $\langle u_i,u_i^*\rangle=1$, then
\begin{itemize}
\item $\frac1{D}\leq \|u_i^*\|\leq \frac{\varrho}{d}$ for $i=1,\ldots,n$;
\item assuming $x=\sum_{i=1}^n \lambda_i u_i$ for $\lambda_i\in\R$, we have
$\lambda_i\leq \frac{\varrho}{d}\cdot\|x\|$.
\end{itemize}
\end{lemma}
\proof
For any $u_i$, let $h_i$ be the distance from $u_i$ to ${\rm lin}\{u_1,\ldots,\hat{u}_i,\ldots,u_n\}$. Then $h_i\leq \|u_i\|$ and
$$
|\det[u_1,\ldots,u_n]| = {\rm det}_{n-1}(u_1,\ldots,\hat{u}_i,\ldots,u_n)\cdot h_i\leq h_i\cdot
\prod_{j\neq i}\|u_j\|.
$$
Since $ \|u_i^*\|=\frac1{h_i}$, we deduce that $\frac1{D}\leq \|u_i^*\|\leq \frac{\varrho}{d}$ for $i=1,\ldots,n$.
In turn, it follows that $\lambda_i=\langle x,u_i^*\rangle\leq \frac{\varrho}{d}\cdot\|x\|$. \qedhere
\endproof

\begin{lemma}
\label{closedeterminant}
If $u_1,\ldots,u_n\in\R^n$ and $v_1,\ldots,v_n\in\R^n$ satisfy that 
$\|u_i\|\leq M$ and $\|u_i-v_i\|\leq \varepsilon$ for $i=1,\ldots,n$,
$M\geq 1$ and $\varepsilon\in(0,1)$, then
$$
\Big|\det[v_1,\ldots,v_n]-\det[u_1,\ldots,u_n]\Big|\leq 2^{n}M^{n-1}\,\varepsilon \,.
$$
\end{lemma}
\proof Using
$\Big|\det[w_1,\ldots,w_n]\Big|\leq \prod_{i=1}^n\|w_i\|$,
$\|v_i\|\leq 2M$ and the linearity of the determinant, we have
\begin{eqnarray*}
\Big|\det[v_1,\ldots,v_n]-\det[u_1,\ldots,u_n]\Big|&=& 
\Big|\det[v_1-u_1,v_2,\ldots,v_n]+ \det[u_1,v_2-u_2,v_3,\ldots,v_n]+\ldots\\
&&\ldots+\det[u_1,\ldots,u_{n-1},v_n-u_n]\Big|\\
&\leq &
\sum_{i=1}^{n} 2^{i-1}M^{n-1}\varepsilon<2^{n}M^{n-1}\cdot \varepsilon. 
\end{eqnarray*}
\endproof

Finally, we use thees two lemmas to estimate how much the pairwise scalar products detetermine a basis, up to congruency, in a quantitative form. 

\begin{lemma}
\label{closeintegral}
For $M>1$ and $n\geq 2$ there exist explicit $\varepsilon_M\in(0,\frac14)$ and $\gamma_M>1$ (depending on $n$ and $M$) with the following properties.
If $u_1,\ldots,u_n\in\R^n$ satisfy $\det[u_1,\ldots,u_n]\geq 1/M$, and $\|u_i\|\leq M$ for $i=1,\ldots,n$, 
and $|\langle u_i,u_j\rangle-k_{ij}|\leq \varepsilon$ and $k_{ij}\in \Z$ for $i,j=1,\ldots,n$  and
$\varepsilon\in(0,\varepsilon_M)$, then there exist
 $v_1,\ldots,v_n\in\R^n$ such that $\|v_i-u_i\|\leq \gamma_M \varepsilon$ 
and $\langle v_i,v_j\rangle=k_{ij}$ for  $i,j=1,\ldots,n$.
\end{lemma}

\noindent {\bf Remark.} The present argument gives $\gamma_M= \sqrt{2} n^{\frac{7}{2} + n} M^2(M^2+1)^{n} .$ 

\begin{proof} By assumption, we have $k_{ij}=k_{ji}$ for $i,j=1,\ldots,n$. Let $A=[u_1,\ldots,u_n]$. The entries $b_{ij}$ of the positive definite symmetric matrix $B=A^tA$ satisfy $|b_{ij}-k_{ij}|\leq \varepsilon$ and $|b_{ij}|\leq M^2$. It follows that each eigenvalue of $B$ is at most $n\cdot M^2$,
which together with $\det B\geq 1/M^2$ yields that each eigenvalue of $M$ is at least $n^{-n}\cdot M^{-2n}$, and moreover each principal minor of $B$ is at least $n^{-n^2} \cdot M^{-2n^2},$ since, by Cauchy's interlacing theorem, all eigenvalues of principal submatrices of $B$ are at least $n^{-n} \cdot M^{-2n}.$ 

Let $K$ be the symmetric matrix $[k_{ij}]_{i,j=1,\dots,n}$. It follows from Lemma~\ref{closedeterminant} that if $\varepsilon_M$ is sufficiently small, then $\det K\geq (M^2+1)^{-1}$ and each principal
minor of $K$ is positive, and hence $K$  is positive definite. In addition
each $|k_{ij}|\leq M^2+1$; therefore, the eigenvalues $\lambda_1\leq \ldots\leq \lambda_n$
 of $K$ satisfy
$$
n^{-n} \cdot (M^2+1)^{-n}\leq \lambda_1\leq \ldots\leq \lambda_n \leq  n(M^2+1).
$$

 We now employ an analysis originally from \cite{Bhatia-Mukherjea}: let $\mathbf{P}(n)$ and $\mathbf{B}(n)$ denote, respectively, the set of $n \times n$ positive definite matrices and  the set of $n \times n$ upper-triangular matrices with positive diagonal elements. If one considers the map $L:\mathbf{P}(n) \to \mathbf{B}(n)$ that associates to a matrix $A \in \mathbf{P}(n)$ its (unique) Cholesky factorisation matrix $L(A) \in \mathbf{B}(n)$ such that $A = L(A)^t L(A),$ then the differential of the map $L$ satisfies (cf. \cite[Equation~(31)]{Bhatia-Mukherjea})
 \[
 \| DL(A)\|_{F} \le \frac{1}{\sqrt{2}} \|A\|_2^{1/2} \|A^{-1}\|_2, 
 \]
 where $\| \cdot \|_{F}$ denotes the Frobenius norm of the operator $DL(A)$ (viewed as an operator from $T_A \mathbf{P}(n)$ to $T_{L(A)} \mathbf{B}(n)$), and $\| \cdot \|_2$ denotes the $2-$norm of a matrix.
 
  The fundamental theorem of calculus then implies that 
 \begin{align}\label{eq:contraction-bound-cholesky}
 \|L(A) - L(B)\|_2 & \le \|A - B\|_2 \sup_{t \in [0,1]} \|DL(tA + (1-t)B)\|_F \cr 
 						  & \le \frac{1}{\sqrt{2}} \|A - B\|_2 \sup_{t \in [0,1]} \|tA + (1-t)B\|_2^{1/2} \|(tA + (1-t)B)^{-1}\|_2. 
 \end{align}
With this tool in hands, notice now that the matrices $B$ and $K$ above are both positive definite, and thus we are in position to use \eqref{eq:contraction-bound-cholesky}, which gives 
\[
\|L(K) - L(B)\|_2 \le \frac{2nM^2}{\sqrt{2}}\|K-B\|_2 \cdot \sup_{t \in [0,1]}\|(tK + (1-t)B)^{-1}\|_2. 
\]
As $(tK+(1-t)B)^{-1}$ is self-adjoint, $\|(tK+(1-t)B)^{-1}\|_2 = \rho((tK+(1-t)B)^{-1}) \le n^n \cdot (M^2 + 1)^{n}.$ We then conclude
\begin{equation}\label{eq:matrix-inequality} 
\|L(K) - L(B)\|_2 \le \sqrt{2} n^{n+1} M^2(M^2+1)^n\|K-B\|_2 \le \sqrt{2} n^{n+2} M^2(M^2+1)^{n} \varepsilon. 
\end{equation}
Write now $Q \cdot L(B) = A,$ for some invertible matrix $Q.$ As $A^t A = L(B)^t L(B),$ it follows that $Q$ is an orthogonal matrix. Define then $[v_1,\dots,v_n] = Q\cdot L(K).$ By \eqref{eq:matrix-inequality} and the fact that $Q$ is orthogonal, we have 
\[
\|[v_1,\dots,v_n] - [u_1,\dots,u_n]\|_{\text{max}} \le \sqrt{n}\|L(K) - L(B)\|_{\text{max}} 
\]
\[
\le \sqrt{n} \|L(K) - L(B)\|_2 \le \sqrt{2} n^{\frac{5}{2} + n} M^2(M^2+1)^{n} \varepsilon.
\]
As $Q \cdot L(B)  = [v_1,\dots,v_n],$
we deduce that 
$$[v_1,\ldots,v_n]^t[v_1,\ldots,v_n]=L(B)^t L(B) = K.$$ 
This concludes the proof. 
\end{proof}

\section{ Stability for lattice packings}\label{sec:ProofLattice}

 The first objective of this section is to provide a short proof of the following weaker result:

\begin{prop}\label{WeakerStab}
	Let $n \in \{8,24\}.$ If $L$ is a lattice in $\R^n$ such that $\lambda(L)\geq \sqrt{2} \cdot \mathbb{I}_{\{n = 8\}} + 2 \cdot \mathbb{I}_{\{n=24\}},$  and $\det L \le 1 + \varepsilon$  for $0<\varepsilon<\varepsilon_n$, then there exist $\Phi\in O(n)$,
	a basis $w_1,\ldots,w_n$  of 
	$$\mathcal{L} = \begin{cases}
		E_8, & \text{ if } n=8, \cr 
		\Lambda_{24}, & \text{ if } n=24, 
	\end{cases}$$
	of length at most $2$ and a basis $u_1,\ldots,u_n$ of~$L$ 
	such that
	$$
	\|u_i-\Phi w_i\|\leq c_n\varepsilon^{1/2}\mbox{ \ for $i=1,\ldots,n$}\,,
	$$
	where $\varepsilon_n, c_n > 0$ are computable parameters. 
\end{prop}

We thus start this section with a statement on the stability of the $E_8$ lattice from its first $2^{16}$ lengths. Concerning that lattice, it has a basis such that each lattice vector in the basis is of length at most $2$. It is also well-known that $E_8$ is the unique even unimodular lattice in $\R^8$; namely, the unique lattice up to orthogonal transformations whose determinant is one, and for which the square length of any lattice vector is an even integer.

\begin{prop}
\label{E8unimodstab}
Let $\varepsilon\in (0,\varepsilon_0)$ with $\varepsilon_0$ sufficiently small. If $L$ is a lattice in $\R^8$ such that
$\det L\leq \frac{10}{9}$, and for any $x\in L\backslash \{o\}$ with $\|x\|\leq 2^{16}$, there exists a positive integer $k$ with
\begin{equation}
\label{E8unimodstabbasiscond}
\left|\,\|x\|^2-2k\right|\leq \varepsilon,
\end{equation}
then there exist a $\Phi\in O(8)$, a basis $u_1,\ldots,u_8$ of $L$, and a basis $z_1,\ldots,z_8$ of $E_8$ with $\max_i\|z_i\|=2$ such that
\begin{equation}
\label{E8unimodstabbasis}
\|u_i-\Phi z_i\|\leq 2^{1000}\cdot \varepsilon \mbox{ \ for }i=1,\ldots,8.
\end{equation}
\end{prop}
\proof 
We observe that $\|x\|\geq 1$ for $x\in L\backslash \{o\}$, and hence an LLL reduced basis
$\tilde{u}_1,\ldots,\tilde{u}_8$ of $L$ with
$\det[\tilde{u}_1,\ldots,\tilde{u}_8]=\det L$
 satisfies that
\begin{equation}
\label{uiupperE8stab}
\|\tilde{u}_i\|\leq 2^{\frac{8\cdot(8-1)}4}\det L< 2^{15}\mbox{ \ for $i=1,\ldots,8$.}
\end{equation}
Since each non-zero vector of $\sqrt{2}\,L$ is of length at least $\sqrt{2}$, Viazovska's theorem yields that
$\det (\sqrt{2}\,L)\geq 1$, and hence 
\begin{equation}
\label{detLlowerE8stab}
\det L\geq 2^{-4}.
\end{equation}

For  $i,j=1,\ldots,8$, \eqref{E8unimodstabbasiscond} yields a positive integer $k_{ii}$ if $i=j$
and positive integer $\tilde{k}_{ij}$ if $i\neq j$ such that
$$
\left|\,\|\tilde{u}_i\|^2-2k_{ii}\right|\leq \varepsilon \mbox{ \ and \ }
\left|\,\|\tilde{u}_i+\tilde{u}_j\|^2-2\tilde{k}_{ij}\right|\leq \varepsilon,
$$
and hence setting $k_{ij}=\tilde{k}_{ij}-k_{ii}-k_{jj}$ if $i\neq j$, we have that
$$
\left|\langle \tilde{u}_i,\tilde{u}_j\rangle -k_{ij}\right|=
\left|\frac{\langle \tilde{u}_i+\tilde{u}_j,\tilde{u}_i+\tilde{u}_j\rangle-\langle \tilde{u}_i,\tilde{u}_i\rangle-\langle \tilde{u}_j,\tilde{u}_j\rangle}2 -
\frac{2\tilde{k}_{ij}-2k_{ii}-2k_{jj}}2\right|\leq \frac32\cdot \varepsilon\leq 2\varepsilon.
$$

It follows from \eqref{uiupperE8stab}, \eqref{detLlowerE8stab} and $\det L\le 2$
that Lemma~\ref{closeintegral} applies with $M=2^{15}$ and using $2\varepsilon$ instead of $\varepsilon$
as $\varepsilon_8<\frac12\,\varepsilon_M$ for the $\varepsilon_M$ of Lemma~\ref{closeintegral}.
Therefore, there exists a basis $\tilde{z}_1,\ldots,\tilde{z}_8$ of $\R^8$ such that
\begin{eqnarray}
\label{E8stabbarz1}
\langle \tilde{z}_i,\tilde{z}_i\rangle&=& 2k_{ii}\mbox{ \ for $i=1,\ldots,8$;}\\
\label{E8stabbarz2}
\langle \tilde{z}_i,\tilde{z}_j\rangle&=& k_{ij}\mbox{ \ for $i,j=1,\ldots,8$ with $i\neq j$;}\\
\label{E8stabbarz3}
\left\| \tilde{u}_i-\tilde{z}_i\right\|&\leq & 
2^{11}M^{2}\cdot 8^8 \cdot(M^2+1)^{8} \cdot 2\varepsilon \le
2^{500}\cdot\varepsilon\mbox{ \ for $i=1,\ldots,8$.}
\end{eqnarray}
It follows from \eqref{E8stabbarz1} and \eqref{E8stabbarz2} that the determinant of the Gram 
matrix associated to $\tilde{z}_1,\ldots,\tilde{z}_8$ is an integer; therefore,
\begin{equation}
\label{tildezdetE8}
\left(\det[\tilde{z_1},\ldots,\tilde{z}_8]\right)^2\in\N\backslash \{0\}.
\end{equation}
Combining Lemma~\ref{closedeterminant} and \eqref{E8stabbarz3} yields that
$$
\Big|\det[\tilde{z}_1,\ldots,\tilde{z}_8]-\det[\tilde{u}_1,\ldots,\tilde{u}_8]\Big|\leq 2^{8}M^{7}\cdot 2^{500}\varepsilon
\le 2^{700}\varepsilon_8<\frac12;
$$
therefore, we deduce from
$\det[\tilde{u}_1,\ldots,\tilde{u}_8]<1+\varepsilon$ and \eqref{tildezdetE8} that
$$
\det[\tilde{z}_1,\ldots,\tilde{z}_8]=1.
$$
Taking also \eqref{E8stabbarz1} into account, we deduce that the lattice generated by $\tilde{z}_1,\ldots,\tilde{z}_8$
is an even unimodular lattice. Therefore, we may asume that $\tilde{z}_1,\ldots,\tilde{z}_8$ is a basis of $E_8$.
As $\tilde{u}_1,\ldots,\tilde{u}_8$ is an LLL basis of $L$,  \eqref{E8stabbarz3} yields that
$$
\prod_{i=1}^8\|\tilde{z}_i\|\leq\left(1+2^{700}\cdot\varepsilon\right)^8\prod_{i=1}^8\|\tilde{u}_i\|\leq \left(1+2^{700}\cdot\varepsilon\right)^8\cdot 2^{\frac{8\cdot(8-1)}4}\det L< 2^{15}.
$$

Let $z_1,\ldots,z_8$ be a basis of $E_8$ such that each $\|z_i\|\leq 2$.
Readily, we have $\|\tilde{z}_i\|>1$, and hence we deduce from Lemma~\ref{dualbasis} that
$z_i=\sum_{j=1}^8\lambda^{(i)}_j\tilde{z}_j$ where
\begin{equation}
\label{lambdaE8}
|\lambda^{(i)}_j|\leq 2^{15}\cdot \|z_i\|\leq 2^{16}
\end{equation}
and the $8\times 8$ matrix $[a_{ij}]$ with $a_{ij}=\lambda^{(i)}_j$ has integer entries and deteminant $\pm 1$.

Finally, we consider the basis $u_i=\sum_{j=1}^8\lambda^{(i)}_j\tilde{u}_j$, $i=1,\ldots,8$ of $L$.
It follows from \eqref{E8stabbarz3} that $i=1,\ldots,8$, then
$$
\left\|u_i-z_i\right\|\leq 8\cdot 2^{16}\cdot
2^{700}\cdot\varepsilon\leq 2^{1000}\varepsilon,
$$
completing the proof of Proposition~\ref{E8unimodstab}.
\endproof

A similar argument proves the following result in dimension 24:

\begin{prop}
\label{Leechunimodstab}
Let $\varepsilon\in (0,\varepsilon_0)$ with $\varepsilon_0$ sufficiently small. If $L$ is a lattice in $\R^{24}$ such that
$\det L\leq \frac{10}{9}$, and for any $x\in L\backslash \{o\}$ with $\|x\|\leq 2^{140}$, there exists an integer $k\geq 2$ with
\begin{equation}
\label{Leechunimodstabbasiscond}
\left|\,\|x\|^2-2k\right|\leq \varepsilon,
\end{equation}
then there exist a $\Phi\in O(24)$, a basis $u_1,\ldots,u_{24}$ of $L$, and a basis $z_1,\ldots,z_{24}$ of the Leech lattice $\Lambda_{24}$ such that
\begin{equation}
\label{Leechunimodstabbasis}
\|u_i-\Phi z_i\|\leq 2^{-10^{10}}\cdot \varepsilon \mbox{ \ for }i=1,\ldots,24.
\end{equation}
\end{prop}
\proof 
We observe that $\|x\|\geq 1$ for $x\in L\backslash \{o\}$, and hence an LLL reduced basis
$\tilde{u}_1,\ldots,\tilde{u}_{24}$ of $L$ with
$\det[\tilde{u}_1,\ldots,\tilde{u}_{24}]=\det L$
 satisfies that
\begin{equation}
\label{uiupperLeechstab}
\|\tilde{u}_i\|\leq 2^{\frac{24\cdot(24-1)}4}\det L< 2^{139}\mbox{ \ for $i=1,\ldots,24$.}
\end{equation}
Since each non-zero vector of $2L$ is of length at least $2$, Cohn-Kumar-Miller-Radchenko-Viazovska theorem yields that
$\det (2\,L)\geq 1$, and hence 
\begin{equation}
\label{detLlowerLeechstab}
\det L\geq 2^{-24}.
\end{equation}

For  $i,j=1,\ldots,24$, \eqref{Leechunimodstabbasiscond} yields a positive integer $k_{ii}$ if $i=j$
and positive integer $\tilde{k}_{ij}$ if $i\neq j$ such that
$$
\left|\,\|\tilde{u}_i\|^2-2k_{ii}\right|\leq \varepsilon \mbox{ \ and \ }
\left|\,\|\tilde{u}_i+\tilde{u}_j\|^2-2\tilde{k}_{ij}\right|\leq \varepsilon,
$$
and hence setting $k_{ij}=\tilde{k}_{ij}-k_{ii}-k_{jj}$ if $i\neq j$, we have that
$$
\left|\langle \tilde{u}_i,\tilde{u}_j\rangle -k_{ij}\right|=
\left|\frac{\langle \tilde{u}_i+\tilde{u}_j,\tilde{u}_i+\tilde{u}_j\rangle-\langle \tilde{u}_i,\tilde{u}_i\rangle-\langle \tilde{u}_j,\tilde{u}_j\rangle}2 -
\frac{2\tilde{k}_{ij}-2k_{ii}-2k_{jj}}2\right|\leq \frac32\cdot \varepsilon\leq 2\varepsilon.
$$

It follows from \eqref{uiupperLeechstab}, \eqref{detLlowerLeechstab} and $\det L<2$
that Lemma~\ref{closeintegral} applies with $M=2^{139}$ and using $2\varepsilon$ instead of $\varepsilon$
as $\varepsilon_{24}<\frac12\,\varepsilon_M$ for the $\varepsilon_M$ of Lemma~\ref{closeintegral}.
Here taking $n=24<2^5$, we have
$$
2\cdot \gamma_M\le 2^{10000}. 
$$
Therefore,  there exists a basis $\tilde{z}_1,\ldots,\tilde{z}_{24}$ of $\R^{24}$ such that
\begin{eqnarray}
\label{Leechstabbarz1}
\langle \tilde{z}_i,\tilde{z}_i\rangle&=& 2k_{ii}\mbox{ \ for $i=1,\ldots,24$;}\\
\label{Leechstabbarz2}
\langle \tilde{z}_i,\tilde{z}_j\rangle&=& k_{ij}\mbox{ \ for $i,j=1,\ldots,24$ with $i\neq j$;}\\
\label{Leechstabbarz3}
\left\| \tilde{u}_i-\tilde{z}_i\right\|&\leq & 
2^{10000}\cdot\varepsilon\mbox{ \ for $i=1,\ldots,24$.}
\end{eqnarray}
It follows from \eqref{Leechstabbarz1} and \eqref{Leechstabbarz2} that the determinant of the Gramm matrix associated to
$\tilde{z}_1,\ldots,\tilde{z}_8$ is an integer; therefore,
\begin{equation}
\label{tildezdetLeech}
\left(\det[\tilde{z_1},\ldots,\tilde{z}_8]\right)^2\in\N\backslash \{0\}.
\end{equation}

Combining Lemma~\ref{closedeterminant} and \eqref{Leechstabbarz3} yields that
$$
\Big|\det[\tilde{z}_1,\ldots,\tilde{z}_{24}]-\det[\tilde{u}_1,\ldots,\tilde{u}_{24}]\Big|\leq 2^{24}M^{23}\cdot 2^{10000}\varepsilon
<2^{15000}\varepsilon_{24}<\frac12;
$$
therefore, we deduce from
$\det[\tilde{u}_1,\ldots,\tilde{u}_{24}]\le\frac{10}{9}$ and \eqref{tildezdetLeech} that
$$
\det[\tilde{z}_1,\ldots,\tilde{z}_{24}]=1.
$$
Taking also \eqref{Leechstabbarz1} into account, we deduce that the lattice $\Lambda$ generated by 
$\tilde{z}_1,\ldots,\tilde{z}_{24}$
is an even unimodular lattice. 
As $\tilde{u}_1,\ldots,\tilde{u}_{24}$ is an LLL basis of $L$,  \eqref{Leechstabbarz3} yields that
$$
\prod_{i=1}^{24}\|\tilde{z}_i\|\leq\left(1+2^{15000}\cdot\varepsilon\right)^{24}\prod_{i=1}^{24}\|\tilde{u}_i\|\leq \left(1+2^{15000}\cdot\varepsilon\right)^{24}\cdot 2^{\frac{24\cdot(24-1)}4}\det L< 2^{139}.
$$

Finally, we show that $\lambda(\Lambda)\geq 2$. Let
$\|z\|\leq 2$ for a $z\in\Lambda\backslash \{o\}$, and let $z=\sum_{j=1}^{24}\lambda_j\tilde{z}_j$ for $\lambda_1,\ldots,\lambda_{24}\in\Z$.
It follows 
from Lemma~\ref{dualbasis} and $\tilde{z}_j>1$ that
$$
|\lambda_j|\leq 2^{139}\cdot \|z\|\leq 2^{140} \mbox{ for $j=1,\ldots,24$,}
$$
and hence $u=\sum_{j=1}^{24}\lambda_j\tilde{u}_j\in L$ satisfies that
$$
\left\|u-z\right\|\leq 24\cdot 2^{140}\cdot
2^{15000}\cdot\varepsilon<\frac{1}{10}
$$
by \eqref{Leechstabbarz3}. Since $\|u\|\geq 2 - 2 \varepsilon$ and $\|z\|\in\sqrt{2}\Z$, we conclude that
$\|z\|\geq 2$, and hence  $\lambda(\Lambda)\geq 2$. We have already verified that $\Lambda$ is even and unimodular; therefore, $\Lambda$ is congruent to $\Lambda_{24}$.
\endproof

We are now ready to prove Proposition \ref{WeakerStab}:

\proof{\it of Proposition \ref{WeakerStab}} We first concentrate on the case of dimension $n=8.$ We assume that $\varepsilon>0$ satisfies
$$
\alpha_8(2^{16})\sqrt{\varepsilon}<\varepsilon_0.
$$

Let $L$ be a lattice in $\R^8$ such $\Delta(L) \ge (1-\varepsilon) \Delta(E_8).$ First, we notice that, by scaling, we may suppose that the packing at hand is of the form $L + \frac{\sqrt{2}}{2}B^n.$ Thus, $\lambda(L)\geq \sqrt{2}.$ In the present case, \eqref{CohnElkies} reads as
$$
1+\sum_{x\in L\backslash\{o\}} g_8(\|x\|)\geq \frac{1}{\det \Lambda}\geq 1-2\varepsilon.
$$

As $\lambda(L) \ge \sqrt{2}$ by definition, and $g_8(t)\leq 0$ for $t\geq \sqrt{2}$, we deduce that
if $x\in L\backslash\{o\}$ satisfies $\|x\|\leq 2^{16}$, then
$0 \ge g_8(\|x\|)\geq -\varepsilon$, and hence Lemma~\ref{tkg8} yields an integer $k\geq 1$ such that
$|\,\|x\|^2-2k|\leq \alpha_8(2^{16})\sqrt{\varepsilon}$. It follows from Proposition~\ref{E8unimodstab} that
there exist a $\Phi\in O(8)$, a basis $u_1,\ldots,u_8$ of $L$, and a basis $z_1,\ldots,z_8$ of $E_8$ with $\max_i\|z_i\|=2$ such that
$$
\|u_i-\Phi z_i\|\leq 2^{1000}\cdot \alpha_8(2^{16})\sqrt{\varepsilon} \mbox{ \ for }i=1,\ldots,8.
$$
This finishes our proof in this case. For the case of dimension 24, one may employ the exact same techniques as above, and therefore we omit the proof. 
\endproof

 We finally outline a proof of theorems \ref{E8stab} and \ref{Leechstab}. We skip many of the details, since they are almost verbatim obtainable from the considerations of Sections 4--8 in \cite{CoK09}. 

\begin{proof}[Outline of proof of Theorems \ref{E8stab} and \ref{Leechstab}]  Again, as in the proof of Proposition \ref{WeakerStab}, we restrict ourselves to one of the cases -- in this case, $n=24,$ since it is also the case dealt with in most of the body of \cite{CoK09}. \\
	
	\noindent\textit{Step 1:} By a Poisson summation argument, as before, we readily obtain that 
	
	$$
	1+\sum_{x\in L\backslash\{o\}} g_{24}(\|x\|)\geq \frac{1}{\det \Lambda}\geq 1-2\varepsilon.
	$$
	Thus, since $\Lambda(L) \ge 2$, if $x\in L\backslash\{o\}$ satisfies $\|x\|\leq 2^{100}$, then
	$0 \ge g_{24}(\|x\|)\geq -\varepsilon$, and hence Lemma~\ref{tkg8} yields an integer $k\geq 1$ such that
	$|\,\|x\|^2-2k|\leq \alpha_{24}(2^{100})\sqrt{\varepsilon}$. Moreover, since $g_{24}$ has a \emph{single} zero at $\|x\| = 2$ -- which follows, for instance, from the proof of the main theorems in \cite{Via17} and \cite{CKMRV17} --, then, if $x \in L \backslash\{o\}$ satisfies $\|x\| \le 2 + 10^{-20},$ then $|\, \|x\|^2 - 4 | \le C \varepsilon.$ \\
	
	\noindent\textit{Step 2:} In light of Section 4 in \cite{CoK09}, one concludes \emph{verbatim} from the proofs of lemmata 4.2 and 4.4 in that paper that, if $\det L \le 1 + \varepsilon,$ where $\varepsilon$ is sufficiently small, the lattice $L$ must have \emph{exactly} 196560 nearly minimal vectors. \\ 
	
	\noindent\textit{Step 3:} Again, a \emph{verbatim} adaptation of Sections 5 and 6 applies to this context; the only change to be made is that, instead of the fixed value of $\varepsilon = 6.733 \cdot 10^{-27}$ taken in \cite{CoK09}, one has to take $\varepsilon >0 $ arbitrary and sufficiently small. By \textit{Step 1} above, this is allowed, since in that step we showed that all the estimates in Section 4 can be obtained with $\varepsilon$ arbitrarily small, and $\mu, \nu, \omega = O(\varepsilon^{1/2})$ in Proposition 4.3 of \cite{CoK09}. \\
	
	\noindent\textit{Step 4:} Using the adapted versions of Sections 4, 5 and 6, we may also use the exact same proofs of lemmata 7.1, 7.2, 7.3, 7.4 in order to obtain Proposition 7.5 in \cite{CoK09}, with, in our case, $\varepsilon>0$ \emph{arbitrary}, as long as it is sufficiently small. Through the arguments from Section 8 in \cite{CoK09}, one readily obtains a basis $u_1, \dots, u_{24}$ of $L$ and a basis $v_1,\dots, v_{24}$ of $\Lambda_{24}$ such that $\|u_i\| < 10$ and $|\langle u_i, u_j \rangle - \langle v_i,v_j \rangle| = O(\varepsilon),\, i,j = 1,\dots,24,$ whenever  $\varepsilon$ is sufficiently small. \\
	
	\noindent\textit{Step 5:} In order to conclude, note that, by Lemma \ref{closeintegral}, there exist vectors  $z_1,\dots,z_{24}$ with $\|u_i - z_i\| \le C \varepsilon, \, i=1,\dots,24$ and $\langle z_i,z_j \rangle = \langle v_i,v_j \rangle, i,j = 1,\dots,24.$ By repeating the same arguments as in the proof of Proposition \ref{Leechunimodstab}, we obtain that the vectors $z_i$ span a lattice congruent to $\Lambda_{24},$ which concludes the desired result in the case of $n=24.$ For $n=8,$ an entirely analogous argument yields the same result. 
\end{proof}

\section{Stability for periodic packings}\label{sec:Periodic}

We will use the following statement about finite packings: 

\begin{lemma}
\label{finitepackigupp}
If $K$ is a compact convex set in $\R^8$ that contains $m\geq 1$ points such that the distance between any two points is at least $\sqrt{2}$, then $m\leq V(K+\sqrt{2}\,B^8)$. Similarly, if $K$ is a compact convex set in $\R^{24}$ that contains $m' \ge 1$ points such that any two points are of distance at least $2$, then $m' \le V(K + 2 B^{24}).$ 
\end{lemma}

\begin{proof}
By directly applying the first upper bound in \cite[Theorem~9.4.1]{Boroczky}, one gets 
\[
\frac{m}{\delta_8\left( \frac{\sqrt{2}}{2} \right)} \le V(K + \sqrt{2}B^8),
\]
where we let $\delta_n(r) = \sup_{\Xi \colon \Xi + r B^n \text{ packing }} \delta(\Xi)$. As $\delta_8(1) = \frac{1}{16},$ by Viazovska's theorem, one obtains that $\delta_8(\sqrt{2}/2) = 1,$ the assertion is proved for dimension 8. For dimension 24, the result is entirely analogous. 
\end{proof}

\proof[Proof of Theorem~\ref{PeriodicPacking8stab0}]  Fix $R:=R_{\varepsilon}$ as in the statement of Theorem \ref{PeriodicPacking8stab0}. Let us consider a periodic packing $\Lambda+\widetilde{S}+\frac{\sqrt{2}}2\,B^8$ of balls of radius $\frac{\sqrt{2}}2$ where we also assume that $(\widetilde{S}-\widetilde{S})\cap \Lambda=\{0\}$ and 
\begin{equation}
\label{LambdaLarge}
\mbox{$\Lambda+2^{21}R\,B^8$ is a packing, and hence $\Lambda+2K$ is a packing.}
\end{equation} 

\noindent\textit{Step 1: discarding a small set.} According to the main result in \cite{Via17}, we have
$$
\frac{\# \widetilde{S}}{\det\Lambda}\leq 1.
$$
By assumption, the packing has almost optimal density; namely
$$
\frac{\# \widetilde{S}}{\det\Lambda}\geq 1-\varepsilon.
$$

We make the packing saturated: we choose
$S\supset\widetilde{S}$ such that  $\Lambda+S+\frac{\sqrt{2}}2\,B^8$ is still a packing, 
$(S-S)\cap \Lambda=\{0\}$
and
 for any $z\in \R^8$ there exists $x\in \Lambda+S$ with $\|x-z\|<\sqrt{2}$. Let $\# S=N$, thus
$$
1-\varepsilon \leq \frac{N}{\det\Lambda}\leq 1,
$$
and hence
$$
N-\# \widetilde{S}\leq \varepsilon\cdot \det\Lambda\leq  2\varepsilon N.
$$

As $g_8(0)=\hat{g}_8(0)=1$, it follows from \eqref{CohnElkies} that
$$
N+\sum_{z\in\Lambda,\;v,w\in S\atop
z+v-w\neq 0} g_8(\|z+v-w\|)\geq \frac{N^2 }{\det \Lambda}\geq (1-\varepsilon)N; 
$$
therefore,
$$
\#\left\{(z,v,w):z\in\Lambda\mbox{ and }v,w\in S
\mbox{ and }g_8(\|z+v-w\|)\leq -\sqrt{\varepsilon}\right\}
\leq N\sqrt{\varepsilon}.
$$
Let $S_0\subset S$ be the set of all $v\in S$ such that there exist $w\in S$ and $z\in \Lambda$ with
$g_8(\|v-(w-z)\|)\leq -\sqrt{\varepsilon}$. It follows that
\begin{eqnarray}
\label{S0upp}
\#S_0&\leq &N\cdot 2\sqrt{\varepsilon}\leq  2\sqrt{\varepsilon} \cdot \det\Lambda,\\
\label{SS0g8}
g_8(\|v-w+z\|)&>& -\sqrt{\varepsilon} \mbox{ \ for $v,w\in S\backslash S_0$ and $z\in\Lambda$.}
\end{eqnarray}

Let $\pi:\R^8\to \R^8/\Lambda$ be the projection map. Hence $\pi$ is injective on $S$ and on $K$ ({\it cf.}\eqref{LambdaLarge}).
Since $V(\R^8/\Lambda)=\det\Lambda$ and $\varepsilon<\frac1{R}$, we observe that 
\begin{eqnarray*}
\int_{\R^8/\Lambda}\#\left(\pi(S)\cap (x+\pi(K))\right)\,dx&=&
\int_{\R^8/\Lambda}\sum_{y\in\pi(S)}\mathbbm{1}_{x+\pi(K)}(y)\,dx=
\int_{\R^8/\Lambda}\sum_{y\in\pi(S)}\mathbbm{1}_{y-\pi(K)}(x)\,dx\\
&=&N\cdot V(K)\geq
(1-\varepsilon)V(\R^8/\Lambda)\cdot V(K)\\
&\geq&
\left(1-\frac1{R}\right)V(\R^8/\Lambda)\cdot V(K).
\end{eqnarray*}
Assuming $R$ is large enough, Lemma~\ref{finitepackigupp} 
and $RB^n\subset K$ yield, for any $x\in \R^8/\Lambda$,
$$
\#\left(\pi(S)\cap (x+\pi(K))\right)\leq V(K+\sqrt{2}\, B^8)\leq
\left(1+\frac{\sqrt{2}}{R}\right)^8\cdot V(K)<
\left(1+\frac{15}{R}\right)\cdot V(K).
$$
It follows from the last two estimates that, for $x\in  \R^8/\Lambda$,
$$
p=\P\left\{\#\left(\pi(S)\cap (x+\pi(K))\right)< \left(1-\frac4{\sqrt{R}}\right)V(K)\mbox{ for }x\in  \R^8/\Lambda\right\}
$$
satisfies
\begin{eqnarray*}
\left(1-\frac1{R}\right) V(K)&\leq &
p\cdot
\left(1-\frac4{\sqrt{R}}\right)V(K)+(1-p)\cdot
\left(1+\frac{15}{R}\right)V(K)\\
&=&\left(1+\frac{15}{R}\right)V(K)-p\cdot \left(\frac4{\sqrt{R}}+\frac{15}{R}\right)V(K)\\
&\leq & \left(1+\frac{15}{R}\right)V(K)-p\cdot \frac4{\sqrt{R}}\cdot V(K);
\end{eqnarray*}
therefore,
\begin{equation}
\label{PSK1}
\P\left\{\#\left(\pi(S)\cap (x+\pi(K))\right)\geq \left(1-\frac4{\sqrt{R}}\right)V(K)
\mbox{ for }x\in  \R^8/\Lambda\right\}\geq 1-\frac4{\sqrt{R}}.
\end{equation}

We observe that if $(x+K)\cap S_0\neq \emptyset$ then $x\in S_0-K$, and
\begin{equation}
\label{PSK2}
V\left(S_0-K\right)\leq 
N\cdot 2\sqrt{\varepsilon}\cdot V(K)\leq
N\cdot 2\sqrt{\varepsilon}\cdot (2^{20}R)^8\leq N\cdot \frac1{\sqrt{R}}
\end{equation}
if $\varepsilon$ is small; therefore, combining \eqref{PSK1} and \eqref{PSK2} yields that
$$
A=\left\{x\in  \R^8:\,\#\left((S+\Lambda)\cap (x+K)\right)\geq \left(1-\frac4{\sqrt{R}}\right)V(K)
\mbox{ and }(S_0+\Lambda)\cap (x+K)=\emptyset
\right\}
$$
satisfies the estimate
$$
\P\left\{x\in A\mbox{ for }x\in  \R^8\right\}=
\P\left\{x\in \pi(A)\mbox{ for }x\in  \R^8/\Lambda
\right\}\geq 1-\frac8{\sqrt{R}}.
$$
As we allow for a small set to be thrown away where we cannot ascertain that any structure will be preserved, we dispose of the complement of the set $A,$ and from now on we focus on this set. 

For $a\in A $, let $s_a\in S+\Lambda$ with $\|a-s_a\|\leq \sqrt{2}$. This exists by the saturatedness of the packing. We claim that there exists
a lattice $L$ depending on $a$ and $S+\Lambda$ with the following properties:
\begin{enumerate}
\item[(A)] For any $v\in (S+\Lambda)\cap (a+K)$, there exists $u\in L$ such that
$\|v-(s_a+u)\|\leq 2^{21}\alpha_8(2^{20}R)\varepsilon^{\frac14}$;
\item[(B)] $2^{-20}\leq \det L\leq 1+\frac8{\sqrt{R}}$;

\item[(C)] $L$ has a basis $w_1,\ldots,w_8$ such that
$\prod_{i=1}^8 \|w_i\|\leq  2^{14}\Big|\det[w_1,\ldots,w_8]\Big|$, and if
$i,j=1,\ldots,8$, then
$\frac18\leq \|w_i\|\leq  2^{56}$ 
and there exists $m_{ij}\in\Z$ with $m_{ii}\in 2\Z$   satisfying
$$
\Big|\langle w_i,w_j\rangle-m_{ij}\Big| \leq 2^{2^{46}}R^3\alpha_8(2^{20}R)\varepsilon^{\frac14}.
$$
\end{enumerate}

To construct $L$, it is convenient to shift $ (S+\Lambda)\cap (a+K)$ in a way such that $s_a$ ends up being the origin; therefore, let
$S'=\Big( (S+\Lambda)\cap (a+K)\Big)-s_a$ and $a'=s+a$, and hence
\begin{equation}
\label{Sprimeaprime}
0\in S'
\mbox{ \ and \ }
S'\subset a'+K \mbox{ where }{\rm diam}\,K\leq 2^{20}R
\mbox{ \ and \ } \# S'\geq \left(1-\frac4{\sqrt{R}}\right)V(K).
\end{equation}
\noindent\textit{Step 2: Constructing a first lattice close by $S'$.}  Since $(S_0+\Lambda)\cap(a+K)=\emptyset$, 
we deduce from \eqref{Sprimeaprime} and Lemma~\ref{tkg8} that  if $p_1,p_2,p_3\in S'$, then
there exist positive integers $k_1,k_2,k_3$ such that
\begin{equation} \label{eq:norms-comparison} 
\Big|\|p_1-p_2\|^2-2k_3\Big|\leq \alpha_8(2^{20}R)\varepsilon^{\frac14}, 
\Big|\|p_2-p_3\|^2-2k_1\Big|\leq \alpha_8(2^{20}R)\varepsilon^{\frac14}
\mbox{ and }
\Big|\|p_3-p_1\|^2-2k_2\Big|\leq \alpha_8(2^{20}R)\varepsilon^{\frac14},
\end{equation}
and hence
\begin{eqnarray}
\nonumber
\left|\langle p_1-p_3,p_2-p_3\rangle-(k_1+k_2-k_3)\right|&=&
\left|\frac{\|p_1-p_3\|^2+\|p_2-p_3\|^2-\|p_1-p_2\|^2}2-\frac{2k_1+2k_2-2k_3}2\right| \\
\label{scalarprodalpha8}
&\leq & 3\alpha_8(2^{20}R)\varepsilon^{\frac14}.
\end{eqnarray}

We now claim that there exist independent vectors $v_1,\ldots,v_8\in S'$ with the following properties:
\begin{enumerate}
\item[(a)] $\sqrt{2}\leq \|v_i\|\leq 3\sqrt{2}< 2^{5/2}$ for $i=0,\ldots,8$;
\item[(b)] $2^4\leq \Big|\det[v_1,\ldots,v_8]\Big|\leq  2^{20}$;
\item[(c)] $\prod_{i=1}^8\|v_i\|\leq 2^{16}\Big|\det[v_1,\ldots,v_8]\Big|$
\end{enumerate}
We construct $v_1,\ldots,v_8$ by induction on $i=1,\ldots,8$. For $v_1$, we choose any 
$v'_1\in\R^8$ with $\|v'_1\|=2\sqrt{2}$, and hence the saturatedness of the packing yields a 
$v_1\in (v'_1+\sqrt{2}\,B^8)\cap S'$. If we have already constructed $v_1,\ldots,v_{i-1}$ for $i\in\{2,\ldots,8\}$, then we take
any $v'_i$ orthogonal to $v_1,\ldots,v_{i-1}$ with $\|v'_i\|=2\sqrt{2}$, and hence the saturatedness of the packing 
again yields a 
$v_i\in (v'_i+\sqrt{2}\,B^8)\cap S'$. 

Now (a) follows from construction, and in turn (a) implies the upper bound in (b). For the lower bound in (b),
we observe that any $v_i$ is at least distance $\sqrt{2}$ from ${\rm lin}\{v_1,\ldots,v_{i-1}\}$ for $i\in\{2,\ldots,8\}$;
therefore, induction on the size of $j$ of $\{i_1,\ldots,i_j\}\subset \{1,\ldots,8\}$ yields  that
${\rm det}_{j}(v_{i_1},\ldots,v_{i_j})\geq \sqrt{2}^j$. Finally, (c) follows from (a) and (b).

Readily, $v_1,\ldots,v_8\in S'$. We consider the lattice
$$
L_0=\{u\in \R^8:\,\langle u,v_i\rangle\in\Z \mbox{ for }i=1,\ldots,8\},
$$
dual to the lattice $\Z v_1+\ldots +\Z v_8$. Choose a basis $u_1,\ldots,u_8\in L_0$ of $L_0$
such that for $i,j=1,\ldots,8$, we have $\langle u_j,v_i\rangle=0$ if $i\neq j$ and 
$\langle u_i,v_i\rangle=1$.  
We observe that Lemma~\ref{dualbasis}, (a) and (c) yield that if $i=1,\ldots,8$, then
\begin{equation}
\label{uinormupp}
\|u_i\|\leq 2^{16}.
\end{equation}

If $x\in\R^8$, then (a) and \eqref{uinormupp}  yield
\begin{eqnarray}
\label{xnormuiviup}
\|x\|&=&\left\|\sum_{i=1}^8 \langle x,v_i\rangle\cdot u_i\right\|\leq 
2^{19}\max_{i=1,\ldots,8}|\langle x,v_i\rangle|\\
\label{xnormuivilow}
\|x\|&\geq& \max_{i=1,\ldots,8}\frac{|\langle x,v_i\rangle|}{\|v_i\|}\geq
\frac18\cdot \max_{i=1,\ldots,8}|\langle x,v_i\rangle|,
\end{eqnarray}
and hence
\begin{eqnarray}
\label{L0normlow}
\|u\|&\geq &\frac18\mbox{ \ for }u\in L_0\backslash\{0\},\\
\label{L0detlow}
 \det L_0&=&\left(\det[v_1,\ldots,v_8]\right)^{-1}\geq 2^{-20}.
\end{eqnarray}

Since $(S_0+\Lambda)\cap (a+K)=\emptyset$, it follows from \eqref{scalarprodalpha8} that for any 
$v\in S'$, there exist integers $\ell_{i}$,  $i=1,\ldots,8$ such that
$|\langle v,v_i\rangle-\ell_{i}|\leq  3\alpha_8(2^{20}R)\varepsilon^{\frac14}$, and hence
\eqref{xnormuiviup} yields that
$u=\sum_{i=1}^8 \ell_{i}u_i\in L_0$ satisfies
\begin{equation}
\label{SclosetoL0}
\|v-u\|\leq 2^{19}\cdot 3\alpha_8(2^{20}R)\varepsilon^{\frac14}\leq 2^{21}\alpha_8(2^{20}R)\varepsilon^{\frac14}.
\end{equation}
In particular, if $i=1,\ldots,8$, then there exists $\tilde{u}_i\in L_0$ with
\begin{equation}
\label{vitildeui}
\|v_i-\tilde{u}_i\|\leq 2^{21}\alpha_8(2^{20}R)\varepsilon^{\frac14}.
\end{equation}

We consider now the sublattice 
$$
\widetilde{L}=\sum_{i=1}^8\Z\,\tilde{u}_i\subset L_0.
$$
It follows from (a), (b), (c) and Lemma~\ref{closedeterminant} that
\begin{enumerate}
\item[(a')] $1\leq \|\tilde{u}_i\|\leq 8$ for $i=0,\ldots,8$;
\item[(b')] $8\leq \Big|\det[\tilde{u}_1,\ldots,\tilde{u}_8]\Big|=\det \widetilde{L}\leq 2^{21}$;
\item[(c')] $\prod_{i=1}^8\|\tilde{u}_i\|\leq 2^{17}\Big|\det[\tilde{u}_1,\ldots,\tilde{u}_8]\Big|$.
\end{enumerate}

Let $\bar{S}\subset L_0$ be the set of all $u\in L_0$ such that
there exists $v\in S'$ with 
$\|v-u\|\leq 2^{21}\alpha_8(R)\varepsilon^{\frac14}$. Thus
$\tilde{u}_1,\ldots,\tilde{u}_8\in \bar{S}$ and
\eqref{SclosetoL0} defines a bijective correspondence between $S'$ and $\bar{S}$
(compare \eqref{L0normlow}). We then finally let $L\subset L_0$ be the
sublattice generated by $\bar{S}$ as an Abelian subgroup.  We deduce from
\eqref{SclosetoL0} that $L$ satisfies (A) and
$$
\widetilde{L}\subset L\subset L_0.
$$

An LLL 
basis $w_1,\ldots,w_8$ of $L$ satisfies $\|w_1\|\cdot\ldots\cdot \|w_8\|\leq 2^{14}\det L$
where $\|w_i\|\geq \frac18$ by \eqref{L0normlow}; therefore,
we deduce from \eqref{L0normlow}, $\det L\geq \det L_0$ and (c') that
\begin{enumerate}
\item[(i)] $2^{-20}\leq \det L\leq \Big|\det[\tilde{u}_1,\ldots,\tilde{u}_8]\Big|
\leq  2^{21}$;
\item[(ii)] $\prod_{i=1}^8 \|w_i\|\leq  2^{14}\Big|\det[w_1,\ldots,w_8]\Big|$;
\item[(iii)] $\frac18\leq \|w_i\|\leq 2^{35}8^7= 2^{56}$ for $i=1,\ldots,8$.
\end{enumerate}
It follows from (i) and (b') that
\begin{equation}
\label{LperL0card}
\#(L/\widetilde{L})=\frac{\det\widetilde{L}}{\det L} \leq 2^{41}.
\end{equation}

\noindent\textit{Step 3: Showing the lattice is close to $E_8.$} Next we claim that for any $u\in L\cap (a'+ K)$, there exist
$s_1,\ldots,s_\ell\in \bar{S}$ and 
$\xi_1,\ldots,\xi_\ell\in \{-1,1\}$
with
\begin{eqnarray}
\label{Laproxsum}
u&=&\sum_{i=1}^\ell \xi_i s_i\\
\label{Lapproxell}
\ell&\leq & 2^{2^{44}}R.
\end{eqnarray}
As a first step towards proving \eqref{Laproxsum} and \eqref{Lapproxell}, we verify that
for any $u\in L\cap (a'+[-R,R]^n)$, there exist
$s'_1,\ldots,s'_{\ell'}\in \bar{S}$ and 
$\xi'_1,\ldots,\xi'_{\ell'}\in \{-1,1\}$ such that

\begin{eqnarray}
\label{Laproxsum0}
\sum_{i=1}^{\ell'} \xi'_i s'_i&\in& u+\widetilde{L}\\
\label{Lapproxell0}
\ell'&\leq &  2^{\#(L/\widetilde{L})-1}.
\end{eqnarray}
Let $\bar{S}^{(1)}$ be the image of $\bar{S}$ in $L/\widetilde{L}$, and hence
$\bar{S}^{(1)}$ generates the Abelian group $L/\widetilde{L}$.
For $i\geq 1$, we define $S^{(i+1)}=S^{(i)}-S^{(i)}$ by induction on $i$. 
In particular, any element of $S^{(i)}$ is of the form
$$
\pm t_1\pm\ldots \pm t_{2^{i-1}}
$$
for $t_1,\ldots, t_{2^{i-1}}\in S^{(1)}$.
If $S^{(i+1)}=S^{(i)}$ for $i\geq 1$,  then 
$S^{(i)}$ is a subgroup of $L/\widetilde{L}$, and hence 
$\bar{S}^{(1)}\subset \bar{S}^{(i)}$ yields that 
$S^{(i)}=L/\widetilde{L}$. Therefore,
$S^{(\#(L/\widetilde{L}))}=L/\widetilde{L}$, completing the proof of 
\eqref{Laproxsum0} and \eqref{Lapproxell0}. 

For any $u\in L\cap (a'+ [-R,R]^8])$, let us consider
$s'_1,\ldots,s'_{\ell'}\in \bar{S}$ and 
$\xi'_1,\ldots,\xi'_{\ell'}\in \{-1,1\}$ provided by  \eqref{Laproxsum0}
 and \eqref{Lapproxell0} such that
$$
u=\sum_{i=1}^{\ell'} \xi'_i s'_i+w
\mbox{ \ where $w\in\widetilde{L}$ and $\ell'\leq  2^{\#(L/\widetilde{L})-1}$. }
$$
As ${\rm diam}\,K\leq 2^{20}R$, we deduce that
$\|w\|\leq  2^{\#(L/\widetilde{L})}\cdot 2^{20}R\leq 2^{2^{42}}R$.
Now $w=\sum_{i=1}^8\lambda_i \tilde{u}_i$ where
Lemma~\ref{dualbasis}, (a') and (c')  yield that
$$
\sum_{i=1}^8|\lambda_i|\leq 8\cdot 2^{17}\|w\| \leq 2^{2^{43}}R
$$
and in turn we conclude the claims \eqref{Laproxsum} and \eqref{Lapproxell}.

Next we claim that if $i,j=1,\ldots,8$, then there exists $m_{ij}\in\Z$ with $m_{ii}\in 2\Z$ such that
\begin{equation}
\label{wiwjsmall}
\Big|\langle w_i,w_j\rangle-m_{ij}\Big| \leq 2^{2^{46}}R^3\alpha_8(2^{20}R)\varepsilon^{\frac14}.
\end{equation}
We deduce from  \eqref{Laproxsum} and \eqref{Lapproxell} that 
if $i=1,\ldots,8$, then there exist
$s_{i1},\ldots,s_{i\ell_i}\in \bar{S}$ and 
$\xi_{i1},\ldots,\xi_{i\ell_i}\in \{-1,1\}$
with
\begin{eqnarray*}
w_i&=&\sum_{r=1}^{\ell_i} \xi_{ir} s_{ir}\\
\ell_i&\leq & 2^{2^{44}}R.
\end{eqnarray*}
Next, the definition of $\bar{S}$ (compare \eqref{SclosetoL0}) yields that 
if $i=1,\ldots,8$ and $r=1,\ldots,\ell_i$,
 then there exists $\sigma_{ir}\in S'$ satisfying
$$
\|s_{ir}-\sigma_{ir}\|\leq 2^{19}\alpha_8(2^{20}R)\varepsilon^{\frac14},
$$
where $\|s_{ir}\|,\|\sigma_{ir}\|\leq 2^{20}R$ (we have assumed that $RB^n\subset K$).
It follows from \eqref{scalarprodalpha8} that
if $i,j=1,\ldots,8$, $r=1,\ldots,\ell_i$ and
$t=1,\ldots,\ell_j$, then there exists
 $k_{ir;jt}\in\Z$ with $k_{ir;ir}\in 2\Z$ such that
$$
\Big|\langle \sigma_{ir},\sigma_{jt}\rangle-k_{ir;jt}\Big|\leq  3\alpha_8(2^{20}R)\varepsilon^{\frac14}.
$$
We deduce that if $i,j=1,\ldots,8$, $r=1,\ldots,\ell_i$ and
$t=1,\ldots,\ell_j$, then
$$\Big|\langle s_{ir},s_{jt}\rangle-k_{ir;jt}\Big|\leq 
\Big|\langle s_{ir},s_{jt}\rangle-\langle \sigma_{ir},\sigma_{jt}\rangle\Big|+
\Big|\langle \sigma_{ir},\sigma_{jt}\rangle-k_{ir;jt}\Big| 
$$
$$\leq  2^{40}R\alpha_8(2^{20}R)\varepsilon^{\frac14}+
3\alpha_8(2^{20}R)\varepsilon^{\frac14}\leq  2^{41}R\alpha_8(2^{20}R)\varepsilon^{\frac14}.
$$
Therefore, if $i,j=1,\ldots,8$, then there exists an integer $m_{ij}$ such that
$$
\Big|\langle w_i,w_j\rangle-m_{ij}\Big| \leq \ell_i\cdot\ell_j\cdot 2^{41}R\alpha_8(2^{20}R)\varepsilon^{\frac14}\leq
2^{2^{46}}R^3\alpha_8(2^{20}R)\varepsilon^{\frac14},
$$
where $m_{ii}$ is even, proving the claim \eqref{wiwjsmall}. In turn, (ii), (iii) and \eqref{wiwjsmall}
imply (C).

The lower bound in (B) follows from the lower bound in (i). To prove the upper bound in (B), 
let $D=\{x\in\R^8:\|x\|\leq \|x-u\|\mbox{ for }u\in L\}$ be the Dirichlet-Voronoi cell of $L$. Hence $L+D$ is a tiling, and
$$
V(D)=\det L.
$$
It follows from the definition of $\bar{S}\subset L$ and the saturatedness of the packing $\Lambda+S$ that
for any $x\in 4B^n$, there exists $u\in \bar{S}$ with $\|x-u\|\leq 2$; therefore,
$$
D\subset 2B^n.
$$
We deduce using \eqref{Sprimeaprime}, $\bar{S}\subset L$ 
and $V(\bar{S}+D)=\#\bar{S}\cdot V(D)$
that
$$
\left(1-\frac4{\sqrt{R}}\right)V(K)\cdot V(D)\leq
V(\bar{S}+D)\leq V\left(K+2B^n\right)\leq \left(1+\frac2{R}\right)^8V(K),
$$
and hence
$$
\Big|\det[w_1,\ldots,w_8]\Big|=\det L=V(D)\leq 1+\frac8{\sqrt{R}}.
$$
With this inequality, we have completed the proof of (A), (B) and (C).\\

\noindent\textit{Step 4: Conclusion.} We deduce from (B), (C), Lemma~\ref{closedeterminant} and Lemma~\ref{closeintegral}
that  there exists a even unimodular lattice $\Lambda_{(a)}$ with basis
$z_1,\ldots,z_8$ such that
$$
\|w_i-z_i\|=
O\Big(R^3\alpha_8(2^{20}R)\varepsilon^{\frac14}\Big) \mbox{ \ for }i=1,\ldots,8.
$$
In particular, in dimension 8 we can promptly conclude that $\Lambda_{(a)}=\Phi_{(a)} E_8$ for an orthogonal transformation
$\Phi_{(a)}\in O(8)$. For any $x\in a'+K$, if $x=\sum_{i=1}^8\lambda_iw_i$, then Lemma~\ref{dualbasis} and (C) yield that
$$
|\lambda_i|= O(R)\mbox{ \ for $i=1,\ldots,8$.}
$$
It follows that there exists
$Z\subset E_8$ such that
$$
d_H(\bar{S},\Phi_{(a)} Z)=O\Big(R^4\alpha_8(2^{20}R)\varepsilon^{\frac14}\Big),
$$
and hence \eqref{SclosetoL0} yields that 
$$
d_H(S',\Phi_{(a)} Z)=O\Big(R^4\alpha_8(2^{20}R)\varepsilon^{\frac14}\Big).
$$
Finally, notice that, from this construction, points in $S'$ are in bijection to those of $\Phi_{(a)}Z.$ Moreover, we know that $\left(1+\frac{10}{R}\right) V(K) \ge \# S' \ge \left( 1 - \frac{4}{\sqrt{R}} \right) V(K),$ and thus we conclude that 
\[
\left|\frac{\#(\Phi_{(a)} Z)}{\#((\Phi_{(a)} E_8) \cap (a+K+\sqrt{2}B^8))} - 1 \right| \le \frac{300}{\sqrt{R}}.
\]
By using the explicit bound on $\alpha_8(\cdot)$ given in Lemma \ref{tkg8} and choosing $R = \frac{\log \left(\frac{1}{\varepsilon}\right)}{\log \log \left( \frac{1}{\varepsilon}\right)},$ we obtain all the claims of Theorem~\ref{PeriodicPacking8stab0}. 
\endproof

\begin{proof}[Proof of Theorem~\ref{PeriodicPacking24stab0}] As the proof of Theorem~\ref{PeriodicPacking24stab0} is, in technical terms, almost identical to that of Theorem~\ref{PeriodicPacking8stab0}, we only highlight the outcome of each step.  Suppose, thus, that $\Lambda + S + B^{24}$ is a packing. Then:\\

\noindent\textit{Step 1: discarding a small set.} In this first step, we also obtain an exceptional set $S_0 \subset S$ such that $\# S_0 \le 2 \det \Lambda \sqrt{\varepsilon}$ and $g_{24}(\|v-w + z\|) > - \sqrt{\varepsilon}$ for $v,w \in S \backslash S_0$ and $z \in \Lambda.$ 
		
Moreover, we also obtain  that the set $A \subset \R^{24}$ where $\#((S + \Lambda) \cap (x + K)) \ge \left( 1 - \frac{4}{\sqrt{R}}\right) V(K)$ and $(S_0 + \Lambda) \cap (x+K) = \emptyset$ for each $x \in A$ has probability at least $1 - \frac{8}{\sqrt{R}}.$ \\
		
\noindent\textit{Step 2: constructing a sublattice close by $S'$.} Subsequently, we may exploit the fact that $ g_{24}(\|v-w + z\|) > - \sqrt{\varepsilon}$ for $v,w \in S \backslash S_0$ and $z \in \Lambda$ together with Lemma \ref{tkg8} in order to conclude that the origin-translated version $S'$ of $(S + \Lambda) \cap (a+K)$ satisfies analogues of \eqref{eq:norms-comparison} and \eqref{scalarprodalpha8}. The crucial difference here is that the integers $k_i,i=1,2,3$ in \eqref{eq:norms-comparison} are taken to be \emph{at least} 2.\\
		
\noindent\textit{Step 3: proving the lattice is close to an even, unimodular lattice in dimension 24.} Analogously as in Step 3 in the proof of Theorem~\ref{PeriodicPacking8stab0} above, one also obtains the existence of a lattice $L$ depending on $a$ and $S+\Lambda$ so that 
		
		\begin{enumerate}
			\item[(A)] For any $v\in (S+\Lambda)\cap (a+K)$, there exists $u\in L$ such that
			$\|v-(s_a+u)\|\leq 2^{10^{10}}\alpha_{24}(2^{10^{10}}R)\varepsilon^{\frac14}$;
			\item[(B)] $2^{-2000}\leq \det L\leq 1+\frac8{\sqrt{R}}$;
			
			\item[(C)] $L$ has a basis $w_1,\ldots,w_{24}$ such that
			$\prod_{i=1}^{24}\|w_i\|\leq  2^{3000}\Big|\det[w_1,\ldots,w_{24}]\Big|$, and if
			$i,j=1,\ldots,24$, then
			$\frac{1}{2^{1000}}\leq \|w_i\|\leq  2^{1000}$ 
			and there exists $m_{ij}\in\Z$ with $m_{ii}\in 2\Z$   satisfying
			$$
			\Big|\langle w_i,w_j\rangle-m_{ij}\Big| \leq 2^{2^{2^{10}}}R^{15}\alpha_{24}(2^{10^{10}}R) \varepsilon^{\frac14}.
			$$
		\end{enumerate}
	\vspace{2mm}

\noindent\textit{Step 4: conclusion.} Finally, one concludes that there is an even, unimodular lattice $\Lambda_{(a)}$ close to the lattice $L$ constructed above, where closeness is to be understood in the sense of two reduced bases of each being close to one another in norm. As in Step 4, we obtain that there exists $X_{(a)} \subset \Lambda_{(a)}$  such that $d_H(S',X_{(a)} \cap (a+K)) = O(R^{20} \alpha_{24}(2^{10^{10}} R)\varepsilon^{\frac{1}{4}} ).$ Moreover, we have that 
	   
	   \begin{equation}\label{eq:X_a-occupy}
	   \left| \frac{\#(X_{(a)} )}{\#(\Lambda_{(a)} \cap (a+K+2B^{24}))} - 1 \right|\le \frac{300}{\sqrt{R}}.
	   \end{equation}

Thus, we must only conclude that $\Lambda_{(a)}$ is congruent to the Leech lattice.  In order to do so, we shall prove that $\lambda(\Lambda_{(a)}) \ge 2,$ which characterizes the Leech lattice uniquely among the 24 (classes of equivalence of) distinct even, unimodular lattices in dimension 24. In fact, suppose not, that is $\lambda(\Lambda_{(a)}) = \sqrt{2}.$ We then claim that there are two vectors $z_1,z_2 \in X_{(a)} \cap (a+K)$ with $\|z_1  - z_2\| = \sqrt{2}.$ 
	   
	   Indeed, if that were not the case, associate to each $z \in X_{(a)} \cap (a+K)$ the set $\mathcal{V}(z) = \{ z' \in \Lambda_{(a)} \colon \|z-z'\| = \sqrt{2} \}.$ By assumption, $\mathcal{V}(z) \neq \emptyset, \forall z \in X_{(a)}.$ 
	   
	   But recall that we have supposed additionally that $\|z_i-z_j\| \ge 2$ whenever $z_i,z_j \in X_{(a)} \cap (a+K).$  Let then $\tilde{z}\in \mathcal{V}(z_1) \cap \mathcal{V}(z_2), z_1,z_2 \in X_{(a)} \cap (a+K).$ Then $|z_1 - z_2| \le 2\sqrt{2}.$ Consider then a sublattice $\Lambda_{(a)}' \subset \Lambda_{(a)}$ such that $\lambda(\Lambda_{(a)}') = 4,$ and $\#(\Lambda_{(a)}/\Lambda_{(a)}') \le 2^{100}.$ Define then $X_{(a)}'$ to be the points in $X_{(a)} \cap \Lambda_{(a)}' \cap (a+K)$ at distance at least $2$ from $\partial(a + K).$ If $\varepsilon$ is small enough, and hence $V(K)$ is sufficiently large, we have
	   \[
	   \#(X_{(a)}') \ge \frac{1}{2^{200}} V(K).
	   \]
	   Notice that now, if $z_1,z_2 \in X_{(a)}',$ the sets $\mathcal{V}(z_1)$ and $\mathcal{V}(z_2)$ are \emph{disjoint.} Moreover, by construction, $\mathcal{V}(z) \subset a+K$ if $z \in X_{(a)}'.$ Hence, by this and the assumption that there are no two points in $X_{(a)}$ at distance $\sqrt{2}$ from one another, 
	   \[
	   \# ((a+K) \cap (\Lambda_{(a)} \setminus X_{(a)})) \ge  \#\left( \bigcup_{z \in X_{(a)}'} \mathcal{V}(z) \right) \ge 2 \#(X_{(a)}') \ge 2^{-199} V(K). 
	   \]
	   For $\varepsilon$ small enough, this contradicts the fact that $X_{(a)}$ occupies almost all of the space of $\Lambda_{(a)}$ in $a+K,$ reflected in \eqref{eq:X_a-occupy}. Thus, there are two points $z_1, z_2 \in X_{(a)} \cap (a+K)$ with $\|z_1-z_2\| = \sqrt{2},$ as claimed. 
	   
	   This, however, leads to another contradiction: for each $z \in X_{(a)} \cap (a+K)$ there is a unique element $v \in S'$ so that $\|v-z\| = O(R^{20} \alpha_{24}(2^{10^{10}} R)\varepsilon^{\frac{1}{4}} ) < \frac{1}{20}.$ Let $v_1,v_2 \in S'$ be the corresponding such elements to $z_1,z_2 \in X_{(a)} \cap (a+K)$ found above. We would have $\|v_1 - v_2\| \le \sqrt{2} + \frac{1}{10},$ while \eqref{eq:norms-comparison} in the 24-dimensional case, as remarked above, gives that $\|v_1 - v_2\| > 2 - \frac{1}{10}$ for $\varepsilon$ sufficiently small. This contradiction stems from supposing that $\lambda(\Lambda_{(a)}) = \sqrt{2},$ hence $\lambda(\Lambda_{(a)}) \ge 2,$ and therefore $\Lambda_{(a)}$ is congruent to the Leech lattice, as desired.  
	
\end{proof} 

\section{Corollaries about bin packings and general packings}
\label{secbin}

This section is devoted to a series of generalizations of our main results, in particular, to the contexts of \emph{bin packings} and \emph{general packings}. We start by discussing a version of our results for bin packings. In more precise words, these are pakings of congruent spheres in a large convex container $C$. Our main result regarding such types of packings is as follows.

\begin{theorem} 
	\label{BinPacking8stab}
	There exist explicitly computable values $\tilde{\varepsilon}_8, \tilde{c}_8 >0$ such for   
	$\varepsilon\in(0,\tilde{\varepsilon}_8)$,
	$R_{\varepsilon} = \frac{|\log \varepsilon|}{\log |\log \varepsilon|}$ and
	a convex body $C$ in $\R^8$
	with $r(C)\geq \varepsilon^{-2}$,
	the following property holds:
	
	If a packing $\Xi+\frac{\sqrt{2}}2\,B^8\subset C$ of balls  satisfies $\# \Xi\geq (1-\varepsilon)V(C)$, and
	$K$ is a convex body containing a ball of radius $R_\varepsilon$ and
	having diameter ${\rm diam}\,K\leq 2^{20}R_\varepsilon$, then 
	with probability at least $1-\tilde{c}_8R_\varepsilon^{-\frac{1}2}$ with respect to the uniform density in $C$, $x\in C$ satisfies that 
	$\# (\Xi\cap (x+K))\geq (1-\tilde{c}_8R_\varepsilon^{-\frac{1}2})V(K)$ and
	$$
	d_H\Big(\Phi Z,\Xi\cap (x+K)\Big)\leq \varepsilon^{\frac{1}{9}}
	\mbox{ where $Z\subset E_8$ and $\Phi$ is an isometry of $\R^8$.}
	$$ 
	Moreover, the set $Z$ has small gaps when compared to a localized version of $E_8,$ in the sense that 
	\[
	\left|\frac{\#(\Phi Z)}{\#((\Phi E_8) \cap (x+K))} - 1 \right| \le \frac{C}{\sqrt{R_{\varepsilon}}},
	\]
	for some absolute computable constant $C>0.$ 
\end{theorem}

\begin{theorem} 
	\label{BinPacking24stab}
	There exist explicitly computable values $\tilde{\varepsilon}_{24}, \tilde{c}_{24} >0$ such for  
	$\varepsilon\in(0,\tilde{\varepsilon}_{24})$, $R_{\varepsilon} = \frac{|\log \varepsilon|}{\log |\log \varepsilon|}$
	and a convex body $C$ in $\R^{24}$  
	with $r(C)\geq \varepsilon^{-2}$, the following property holds:
	
	If a packing $\Xi+B^{24}\subset C$ of balls  satisfies $\# \Xi\geq (1-\varepsilon)V(C)$, and
	$K$ is a convex body containing a ball of radius $R_\varepsilon$ and
	having diameter ${\rm diam}\,K\leq 2^{140}R_\varepsilon$, then 
	with probability at least $1-\tilde{c}_{24}R_\varepsilon^{-\frac{1}2}$ with respect to the uniform density in $C$, $x\in C$ satisfies that 
	$\# (\Xi\cap (x+K))\geq (1-\tilde{c}_{24}R_\varepsilon^{-\frac{1}2})V(K)$ and
	$$
	d_H\Big(\Phi Z,\Xi\cap (x+K)\Big)\leq \varepsilon^{\frac{1}{9}}
	\mbox{ where $Z\subset \Lambda_{24}$ and $\Phi$ is an isometry of $\R^{24}$.}
	$$ 
	 Moreover, the set $Z$ has small gaps when compared to a localized version of $\Lambda_{24},$ in the sense that 
	\[
	\left|\frac{\#(\Phi Z)}{\#((\Phi \Lambda_{24}) \cap (x+K))} - 1 \right| \le \frac{C'}{\sqrt{R_{\varepsilon}}},
	\]
	where $C' > 0$ is an absolute computable constant. 
\end{theorem}

\begin{remark}
	More precisely, $d_H\Big(\Phi Z,\Xi\cap (x+K)\Big)\leq \tilde{c}_8 e^{\frac{1}{\tilde{c}_{8}} R_{\varepsilon}} \varepsilon^{\frac{1}{8}}=\varepsilon^{\frac{1}{8}-o(1)}$ holds in Theorem \ref{BinPacking8stab} as $\varepsilon$ goes to zero. Analogously, $d_H\Big(\Phi Z,\Xi\cap (x+K)\Big)\leq \tilde{c}_{24} e^{\frac{1}{\tilde{c}_{24}} R_{\varepsilon}} \varepsilon^{\frac{1}{8}}=\varepsilon^{\frac{1}{8}-o(1)}$ also holds in Theorem \ref{BinPacking24stab} as $\varepsilon \to 0.$

	It follows from \eqref{denseinCE8} (resp \eqref{denseinCLeech}) that Theorem~\ref{BinPacking8stab} (resp. Theorem \ref{BinPacking24stab}) applies to the densest (finite) packing of balls of radius $\frac{\sqrt{2}}2$ into $C$.
\end{remark}

Before moving on to the proofs of Theorems \ref{BinPacking8stab} and \ref{BinPacking24stab}, we first prove that their statement are, in fact, \emph{nontrivial}. This is the content of the following proposition  

\begin{prop} The following assertions hold. 
	\begin{enumerate}
		\item Given $C\subset \R^8$ with inradius $r(C)\geq 16$, there exists a finite packing $\Xi+\frac{\sqrt{2}}2\,B^8\subset C$ with
		\begin{equation}
			\label{denseinCE8}
			\#\Xi\geq \left(1-\frac{8}{r(C)}\right)V(C);
		\end{equation}
	   \item Given
	   $C\subset \R^{24}$ with inradius $r(C)\geq 48$,
	   there exists a finite packing $\Xi+B^{24}\subset C$ with
	   \begin{equation}
	   	\label{denseinCLeech}
	   	\#\Xi\geq \left(1-\frac{24}{r(C)}\right)V(C).
	   \end{equation} 
	\end{enumerate}
\end{prop}
\begin{proof}
We note that for any finite packing $\Xi+\frac{\sqrt{2}}2\,B^8\subset C$, we have
$\#\Xi\leq \left(1+\frac{8}{r(C)}\right)V(C)$ by Lemma~\ref{finitepackigupp}. Similarly, for any finite packing $\Xi+B^{24}\subset C$, we have
$\#\Xi\leq \left(1+\frac{24}{r(C)}\right)V(C)$ by Lemma~\ref{finitepackigupp}.

In the $8$-dimensional case, we may assume that $r(C)\cdot B^8\subset C$, and hence $C_0= (1-\frac1{r(C)})C$ satisfies
	$$
	C_0+B^8\subset C.
	$$
	Now, there exists some $x\in\R^8$ such that 
	$$
	\#\Big(C_0\cap (E_8-x)\Big)\geq V(C_0)=\left(1-\frac1{r(C)}\right)^8V(C)\geq \left(1-\frac8{r(C)}\right)V(C);
	$$
	thus $\Xi=C_0\cap (E_8-x)$ is a suitable set for the packing construction in dimension 8. The 24-dimensional argument is entirely analogous. 
\end{proof} 

We now present a proof of Theorems \ref{BinPacking8stab} and \ref{BinPacking24stab}. As the proof in the 24-dimensional context is almost identical to the 8-dimensional one, we decided to omit the former and only include the latter. 

\begin{proof}[Proof of Theorems \ref{BinPacking8stab} and \ref{BinPacking24stab}] We first consider auxiliary cubes of edge length $\varepsilon^{-1}$  in $C$; namely, let
	$$
	\Theta=\{z\in \varepsilon^{-1}\Z^8\colon z+[0,\varepsilon^{-1}]^8\subset C\},
	$$
	and hence, as $r(C)\geq \varepsilon^{-2}$,
	\begin{equation}
		\label{Thetaeps}
		\sum_{z\in\Theta}\#\Big((x+[0,\varepsilon^{-1})^8)\cap \Xi\Big)\geq (1-O(\varepsilon))V(C).
	\end{equation}
	For the following set of regular translates in $\Theta,$
	$$
	\Theta_0=\left\{z\in \Theta\colon \#\Big((z+[0,\varepsilon^{-1})^8)\cap \Xi\Big)\geq (1-\sqrt{\varepsilon})
	V\Big([0,\varepsilon^{-1})^8\Big)\right\},
	$$
	an analogous argument as that in Theorem~\ref{PeriodicPacking8stab0} shows that
	\begin{equation}
		\label{Theta0eps}
		\#\Theta_0\geq \Big(1-O\Big(\sqrt{\varepsilon}\Big)\Big)\cdot \#\Theta.
	\end{equation}
	
	For a fixed $z\in\Theta_0$, we consider the periodic sphere packing
	$$
	\left(\sqrt{2}+\varepsilon^{-1}\right)\cdot\Z^8+\left(((z+[0,\varepsilon^{-1})^8)\cap \Xi\right)+\frac{\sqrt{2}}2\,B^8=
	\Xi_z+\frac{\sqrt{2}}2\,B^8,
	$$ 
	that has center density at least $(1-\sqrt{\varepsilon})\left(\frac{\varepsilon^{-1}}{\sqrt{2}+\varepsilon^{-1}}\right)^8 >1-2\sqrt{\varepsilon}$, as long as $\varepsilon$ is sufficiently small. 
	
	Now, we wish to apply Theorem~\ref{PeriodicPacking8stab0}
	to the periodic sphere packing $\Xi_z+\frac{\sqrt{2}}2\,B^8$.
	We observe that $\lim_{\varepsilon\to 0^+}\frac{R_{2\sqrt{\varepsilon}}}{R_\varepsilon}=\frac12$; therefore, we may safely use $R_\varepsilon$ in place of $R_{2\sqrt{\varepsilon}}$.	It follows thus from Theorem~\ref{PeriodicPacking8stab0} that
	with probability at least $1-O\left(R_\varepsilon^{-\frac{1}2}\right)$, $x\in\R^8$ satisfies that 
	$\# (\Xi_z\cap (x+K))\geq \Big(1-O\Big(R_\varepsilon^{-\frac{1}2}\Big)\Big)V(K)$ and
	there exist a $Z\subset E_8$ and an isometry $\Phi$ of $\R^8$ such that
	$$
	d_H\Big(\Phi Z,\Xi_z\cap (x+K)\Big)=
	O\Big( e^{\frac{1}{\tilde{c}_{8}} R_{\varepsilon}} \sqrt{\varepsilon}^{\frac{1}{4}}\Big)=
	O\Big( e^{\frac{1}{\tilde{c}_{8}} R_{\varepsilon}} \varepsilon^{\frac{1}{8}}\Big).
	$$ 
	As $\varepsilon R_\varepsilon<R_\varepsilon^{-\frac{1}2}$, we deduce that
	with probability at least $1-O\left(R_\varepsilon^{-\frac{1}2}\right)$ with respect to the uniform probability measure on 
	$z+[0,\varepsilon^{-1})^8$, $x\in z+[0,\varepsilon^{-1})^8$ satisfies that 
	$$
	x+K\subset{\rm int}\left(z+[0,\varepsilon^{-1})^8\right),
	$$
	$\# (\Xi\cap (x+K))\geq \Big(1-O\Big(R_\varepsilon^{-\frac{1}2}\Big)\Big)V(K)$ and
	there exist a $Z\subset E_8$ and an isometry $\Phi$ of $\R^8$ such that
	\begin{equation}
		\label{withinTheta0eps}
		d_H\Big(\Phi Z,\Xi\cap (x+K)\Big)=
		O\Big( e^{\frac{1}{\tilde{c}_{8}} R_{\varepsilon}} \varepsilon^{\frac{1}{8}}\Big).
	\end{equation}
	The small gap assertion on $Z$ follows directly from the application of Theorem \ref{PeriodicPacking8stab0}. We observe that for any absolute constant $c>0$, there exists $\varepsilon_0$ such that
	$c\cdot  e^{\frac{1}{\tilde{c}_{8}} R_{\varepsilon}} \varepsilon^{\frac{1}{8}}<\varepsilon^{\frac{1}{9}}$
	if $\varepsilon\in(0,\varepsilon_0)$. 
	Therefore, combining \eqref{Thetaeps}, \eqref{Theta0eps} and \eqref{withinTheta0eps} yields 
 Theorems \ref{BinPacking8stab} and \ref{BinPacking24stab}.
\end{proof}

Finally, we comment on \emph{general packings}. In that case, the statements are even less precise because arbitrarily large holes can be left out even from the densest possible packing. 
According to Groemer \cite{Gro63}, however, the upper density of any general packing can be approximated by densities of periodic packings.
In particular, if $\Xi+\frac{\sqrt{2}}2\,B^8$ is a general packing, then Theorem~\ref{E8} yields that
$\Delta_{\rm upp}(\Xi,\frac{\sqrt{2}}2\,B^8)\leq 1$. Similarly, if $\Xi+B^{24}$ is a general packing, then  Theorem~\ref{Leech} implies that
$\Delta_{\rm upp}(\Xi,B^{24})\leq 1$.

Suppose now $\Xi+\frac{\sqrt{2}}2\,B^8$ is a general sphere packing in dimension 8, such that $\Delta_{\rm upp}(\Xi, \frac{\sqrt{2}}{2} B^8) \ge 1-\varepsilon.$ By the aforementioned result in \cite{Gro63}, rhere exists a sequence of radii $\varrho_i>1$,
$i=1,2,\ldots$ with $\varrho_i\to \infty$ such that for each $\varrho_i$ sufficiently large, we have
$$
\#(\Xi\cap \varrho_iB^8)>(1-2\varepsilon)V(\varrho_iB^8).
$$
This shows that the bin packing $\Xi \cap \varrho_i B^8$ satisfies the hypotheses of Theorem~\ref{BinPacking8stab} with $2\varepsilon,$ as long as $\varrho_i > R_{2\varepsilon}^2.$ Applying that result, we obtain directly the following theorem:

\begin{theorem} 
	\label{GeneralPacking8stab}
	There exist explicitly computable values $\tilde{\varepsilon}_8, \tilde{c}_8 >0$ such for  any 
	$\varepsilon\in(0,\tilde{\varepsilon}_8)$, if
	$R_{\varepsilon} = \frac{|\log \varepsilon|}{\log |\log \varepsilon|}$, then the following property holds:
	
	If a packing $\Xi+\frac{\sqrt{2}}2\,B^8\subset \R^8$ of balls  satisfies $\Delta_{\rm upp}(\Xi,B^8)\geq 1-\varepsilon$, and
	$K$ is a centered convex body containing a ball of radius $R_\varepsilon$ and
	having diameter ${\rm diam}\,K\leq 2^{20}R_\varepsilon$, then there exists
	a sequence of radii $\varrho_i>1$,
	$i=1,2,\ldots$ with $\varrho_i\to \infty$ such that for each $\varrho_i$,
	with probability at least $1-\tilde{c}_8R_\varepsilon^{-\frac{1}2}$ with respect to the uniform density in $\varrho_iB^8$, 
	$x\in \varrho_iB^8$ satisfies that 
	$\# (\Xi\cap (x+K))\geq (1-\tilde{c}_8R_\varepsilon^{-\frac{1}2})V(K)$ and
	$$
	d_H\Big(\Phi Z_i,\Xi\cap (x+K)\Big)\leq \varepsilon^{\frac{1}{9}}
	\mbox{ where $Z_i\subset E_8$ and $\Phi$ is an isometry of $\R^8$.}
	$$ 
	Moreover, the sets $Z_i$ have small gaps when compared to a localized version of $E_8,$ in the sense that 
	\[
	\left|\frac{\#(\Phi Z_i)}{\#((\Phi E_8) \cap (x+K))} - 1 \right| \le \frac{C}{\sqrt{R_{\varepsilon}}},
	\]
	for some absolute computable constant $C>0.$ 
\end{theorem}

The proof above may be completely adapted to the 24 dimensional case, which yields the following result on stability of general packings in dimension 24: 

\begin{theorem} 
	\label{GeneralPacking24stab}
	There exist explicitly computable values $\tilde{\varepsilon}_{24}, \tilde{c}_{24} >0$ such for  any 
	$\varepsilon\in(0,\tilde{\varepsilon}_{24})$, if
	$R_{\varepsilon} = \frac{|\log \varepsilon|}{\log |\log \varepsilon|}$, then the following property holds:
	
	If a packing $\Xi+B^{24}\subset \R^{24}$ of balls  satisfies $\Delta_{\rm upp}(\Xi,B^{24})\geq 1-\varepsilon$, and
	$K$ is a centered convex body containing a ball of radius $R_\varepsilon$ and
	having diameter ${\rm diam}\,K\leq 2^{140}R_\varepsilon$, then 
	there exists
	a sequence of radii $\varrho_i>1$,
	$i=1,2,\ldots$ with $\varrho_i\to \infty$ such that for each $\varrho_i$ sufficiently large, with probability at least $1-\tilde{c}_{24}R_\varepsilon^{-\frac{1}2}$ with respect to the uniform density in $\varrho_i B$, $x\in \varrho_iB$ satisfies that 
	$\# (\Xi\cap (x+K))\geq (1-\tilde{c}_{24}R_\varepsilon^{-\frac{1}2})V(K)$ and
	$$
	d_H\Big(\Phi Z_i,\Xi\cap (x+K)\Big)\leq \varepsilon^{\frac{1}{9}}
	\mbox{ where $Z_i\subset \Lambda_{24}$ and $\Phi$ is an isometry of $\R^{24}$.}
	$$ 
	Moreover, the sets $Z_i$ have small gaps when compared to a localized version of $\Lambda_{24},$ in the sense that 
	\[
	\left|\frac{\#(\Phi Z_i)}{\#((\Phi \Lambda_{24}) \cap (x+K))} - 1 \right| \le \frac{C'}{\sqrt{R_{\varepsilon}}},
	\]
	where $C' > 0$ is an absolute computable constant. 
\end{theorem}

\section*{Acknowledgements} 

J.P.G.R. acknowledges support by the ERC grant RSPDE 721675, and K.J.B. ackowledges the hospitality of FIM at ETH Z\"urich where part of the research was done, and the support by NKFIH grant 132002.  All authors would like to express their deepest gratitude towards the anonymous referee  for indicating how to prove the current of version of Theorems~\ref{E8stab} and \ref{Leechstab} through adapting the techniques from \cite{CoK09}.

\end{document}